\documentclass[11pt]{amsart}
\usepackage{palatino, mathpazo}
\usepackage{amscd}      
\usepackage{amssymb}
\usepackage{amsmath}
\usepackage{dsfont}
\usepackage{xypic}      
\LaTeXdiagrams          
\usepackage[all,v2]{xy}
\xyoption{2cell} \UseAllTwocells \xyoption{frame} \CompileMatrices

\usepackage{latexsym}
\usepackage{epsfig}
\usepackage{amsfonts}
\usepackage{enumerate}
\usepackage{mathrsfs}

\newtheorem{prop}{Proposition}[section]
\newtheorem{lem}[prop]{Lemma}

\newtheorem{cor}[prop]{Corollary}
\newtheorem{thm}[prop]{Theorem}
\newtheorem{rmk}[prop]{Remark}
\newtheorem{example}[prop]{Example}

\newtheorem{defn}[prop]{Definition}

\newenvironment{pf}{\begin{trivlist}\item[]{\sc Proof.}}%
            {\nolinebreak $\Box$ \end{trivlist}}

\newcommand{\noprint}[1]{}

\newcommand{\zz}{{\mathbb Z}}

\newcommand{\aaa}{{\mathbb A}}

\newcommand{\dd}{{\mathbb D}}

\renewcommand{\ll}{{\mathbb L}}
\newcommand{\qq}{{\mathbb Q}}
\newcommand{\pp}{{\mathbb P}}
\newcommand{\cc}{{\mathbb C}}

\newcommand{\sM}{\mathcal{M}}
\newcommand{\sO}{\mathcal{O}}
\newcommand{\sX}{\mathcal{X}}
\newcommand{\sT}{\mathcal{T}}
\newcommand{\sY}{\mathcal{Y}}
\newcommand{\sZ}{\mathcal{Z}}
\newcommand{\sF}{\mathcal{F}}
\newcommand{\sC}{\mathcal{C}}

\newcommand{\sP}{\mathcal{P}}
\newcommand{\sA}{\mathcal{A}}
\newcommand{\sB}{\mathcal{B}}
\newcommand{\sQ}{\mathcal{Q}}

\newcommand{\bP}{\mathbb{P}}

\newcommand{\sTT}{{~^{0}\sT}}
\newcommand{\ssT}{{~^{-1}\sT}}
\newcommand{\sFF}{{~^{0}\sF}}
\newcommand{\ssF}{{~^{-1}\sF}}
\newcommand{\sPA}{{~^{p}\sA}}
\newcommand{\sPF}{{~^{p}\sF}}
\newcommand{\sPT}{{~^{p}\sT}}
\newcommand{\sQT}{{~^{q}\sT}}
\newcommand{\sPB}{{~^{p}\sB}}
\newcommand{\sPM}{{~^{p}\sM}}
\newcommand{\srH}{{\mathscr H}}
\newcommand{\srT}{{~^{p}\mathscr T}}
\newcommand{\srPH}{{~^{p}\srH}}
\newcommand{\smQT}{{~^{q}\mathfrak{T}}}
\newcommand{\smPF}{{~^{p}\mathfrak{F}}}
\newcommand{\pDelta}{{~^{p}\Delta}}

\DeclareMathOperator{\spec}{Spec}
\DeclareMathOperator{\Coh}{Coh}

\DeclareMathOperator{\Hom}{Hom}
\DeclareMathOperator{\Per}{Per}
\DeclareMathOperator{\Ch}{Ch}
\DeclareMathOperator{\Hilb}{Hilb}
\DeclareMathOperator{\PHilb}{~^{p}Hilb}
\DeclareMathOperator{\Supp}{Supp}
\DeclareMathOperator{\Coker}{Coker}
\DeclareMathOperator{\pCoker}{~^{p}Coker}
\DeclareMathOperator{\DT}{DT}
\DeclareMathOperator{\PT}{PT}
\DeclareMathOperator{\pDT}{~^{p}DT}
\DeclareMathOperator{\St}{St}
\DeclareMathOperator{\Sch}{Sch}
\DeclareMathOperator{\reg}{\tiny reg}
\DeclareMathOperator{\ssc}{\tiny{sc}}
\DeclareMathOperator{\sur}{\tiny{sur}}
\DeclareMathOperator{\exc}{\tiny{exc}}
\DeclareMathOperator{\Cone}{Cone}
\DeclareMathOperator{\Ad}{Ad}
\DeclareMathOperator{\FM}{FM}
\DeclareMathOperator{\red}{red}
\DeclareMathOperator{\id}{id}
\DeclareMathOperator{\Ext}{Ext}

\newcommand{\age}{\mathop{\rm age}}

\def\Label#1{\label{#1}{\tt [#1]}\phantom{h}}
\def\Label{\label}

\title[Donaldson-Thomas invariants under flops]{Donaldson-Thomas invariants of Calabi-Yau orbifolds under flops}
\author{Yunfeng Jiang}

\address{Department of Mathematics\\ University of Kansas\\ 405 Snow Hall, 1460 Jayhawk Blvd\\Lawrence, KS 66045}
\email{y.jiang@ku.edu}

\begin{document}
\sloppy \maketitle
\begin{abstract}
We study the Donaldson-Thomas type invariants for  the Calabi-Yau  threefold Deligne-Mumford stacks under flops. 
A crepant birational morphism between two smooth Calabi-Yau  threefold Deligne-Mumford stacks is called an orbifold flop if the flopping locus is the quotient of weighted projective lines by a cyclic group action. We prove that the Donaldson-Thomas invariants are preserved under orbifold flops.
\end{abstract}

\tableofcontents

\maketitle

\section{Introduction}

The goal of this paper is to prove a natural property that the Donaldson-Thomas (DT) type invariants of Calabi-Yau threefold Deligne-Mumford (DM)  stacks are preserved under orbifold flops. The techniques we use are Bridgeland's Hall algebra identities inside the motivic Hall algebra of some abelian categories, Joyce-Song's integration map from the motivic Hall algebra to the ring of functions on the  quantum torus, and Calabrese's method of Hall algebra identities under threefold  flops.

\subsection{Motivation and the DT-invariants}

Let $X$ be a proper smooth Calabi-Yau threefold. Fixing the topological data $(\beta,n)$ for $\beta\in H_2(X,\zz)$, and $n\in \zz$, the DT invariant $\DT_{n,\beta}$ is defined by the virtual count of the Hilbert scheme of curves on $X$ with topological data $(\beta, n)$:
$$\DT_{n,\beta}=\int_{[I_n(X,\beta)]^{virt}}1,$$
where $I_n(X,\beta)$ is the Hilbert scheme of curves $C$ on $X$ (the Donaldson-Thomas moduli space) such that 
$$[C]=\beta, \chi(\sO_{C})=n.$$
Here $[I_n(X,\beta)]^{virt}$ is the zero dimensional  virtual fundamental class of $I_n(X,\beta)$, constructed by R. Thomas in \cite{Thomas} since  the scheme $I_n(X,\beta)$ admits a perfect obstruction theory in the sense of \cite{BF}, \cite{LT}. 

In \cite{Behrend}, Behrend provides another way to the DT-invariants of Calabi-Yau threefolds, which are not necessarily proper. 
Behrend proves that the scheme  $I_n(X,\beta)$ admits a symmetric obstruction theory and if it is proper, the  virtual count is given by the weighted Euler characteristic:
$$\DT_{n,\beta}=\int_{[I_n(X,\beta)]^{virt}}1=\chi(I_n(X,\beta), \nu_I),$$
where $\nu_{I}: I_n(X,\beta)\to \zz$ is an integer valued constructible function which we  call the {\em Behrend function} of $I_n(X,\beta)$.  Behrend's theory works for any moduli schemes of objects on the derived category of coherent sheaves $D^b(X)$ admitting a symmetric obstruction theory.  This makes the DT-invariants into motivic invariants. 

A very important variation of DT-invariant is the Pandharipande-Thomas (PT) stable pair invariant.  
\begin{defn}\label{PT-stable:pair}(\cite{PT})
A stable pair 
$[\sO_{X}\stackrel{s}{\rightarrow} F]$ is a two-term complex in $D^b(X)$ satisfying:
\begin{enumerate}
\item $\dim(F)\leq 1$ and $F$ is pure;
\item $s$ has zero-dimensional cokernel.
\end{enumerate}
\end{defn}
The moduli scheme $\PT_n(X,\beta)$ of stable pairs with fixing topological data 
$[F]=\beta\in H_2(X,\zz), \chi(F)=n$ is  a scheme and the PT-invariant is defined by 
$$\PT_{n,\beta}(X)=\chi(\PT_n(X,\beta), \nu_{\PT}),$$
where $\nu_{\PT}$ is the Behrend function on  $\PT_n(X,\beta)$.  
Both DT-invariants and PT-invariants are curve counting invariants of $X$; The famous DT/PT-correspondence conjecture in  \cite{PT} equates these two invariants in terms of partition functions. 

The conjecture was proved by Bridgeland \cite{Bridgeland11}, and Toda \cite{Toda10} using the wall crossing idea, under which the DT-moduli space and the PT moduli space correspond to different (limit) stability conditions in Bridgeland's space of stability conditions. 

We pay more attention to Bridgeland's method for the proof.  In \cite{Bridgeland11} Bridgeland uses some Hall algebra identities in the motivic Hall algebra $H(\sA)$ of the abelian category of coherent sheaves $\sA=\Coh(X)$; such that the DT-moduli space and the PT-moduli space are both elements in the Hall algebra. Then applying the integration map as in \cite{JS}, \cite{Bridgeland10} Bridgeland gets the DT/PT-correspondence. 

The same idea works for threefold flops.  Let 
$$
\xymatrix{
X\ar@{-->}[rr]^{\phi}\ar[dr]_{\psi}&& X^\prime\ar[dl]^{\psi^\prime}\\
&Y&
}
$$
be a flopping contraction, such that $\psi, \psi^\prime$ all contract rational curves $\pp^1$ to singular points, with normal bundle 
$\sO_{\pp^1}(-1)\oplus \sO_{\pp^1}(-1)$.  The local model is the Atiyah flop. 
In  \cite{Calabrese} Calabrese studies and proves the flop formula of the DT-type invariants using the method of the Hall algebra identities and the integration map, generalizing the idea in \cite{Bridgeland11}. 

More precisely, for the flop $\phi: X\dasharrow X^\prime$, Bridgeland \cite{Bridgeland} proves that their derived categories
are equivalent:
$$\Phi: D^b(X)\to D^b(X^\prime)$$
where $\Phi$ is given by the Fourier-Mukai type transformation.  Furthermore, the equivalence 
$\Phi$ sends the category of perverse sheaves to perverse sheaves, i.e.
$$\Phi(^{q}\Per(X))=~^{p}\Per(X^\prime),$$
where $q=-(p+1)$ is the perversity.  Usually we take $p=-1, 0$. 
In \cite{Calabrese} Calabrese proves some Hall algebra identities in the Hall algebra $H(\sPA)$ for $\sPA:=~^{p}\Per(X)$.  Since 
$\Phi$ preserves perverse sheaves, applying the integration map he gets the flop formula for the DT-invariants. 
A proof of the flop formula for DT-type invariants using Joyce's wall crossing is given by Toda in \cite{Toda}; and the study of DT-invariants under blow-ups and flops using J. Li's degeneration formula was given by Hu-Li \cite{HL}. 

\subsection{Flops of Calabi-Yau threefold stacks}

In this paper we consider the orbifold flop of Calabi-Yau threefold DM  stacks.  The reason to consider Calabi-Yau  threefold stacks (or orbifolds), on one hand, is that in  general  there exists a global  K\"ahler moduli space,  and there are two large volume points: one corresponds to the crepant resolution of the orbifold singularity, and one corresponds to the orbifold singularity. 
They are usually derived equivalent, hence some information in the orbifold side determines the side of the crepant resolution. 
On the other hand, ``The Crepant Transformation Conjecture" (CTC) of Y. Ruan in Gromov-Witten (GW) theory has been attracted a lot of interests, see \cite{LR}, \cite{LLW}, \cite{CIJ}, \cite{CJL}.  It is interesting to consider the general CTC conjecture for DT-theory. Also this gives us a chance to learn Bridgeland's method of Hall algebra identities. 

The ``Crepant Resolution Conjecture" for the DT-invariants was formulated by J. Bryan etc in \cite{BCY} for Calabi-Yau orbifolds satisfying  the Hard Lefschetz (HL) conditions. In \cite{Calabrese2},  Calabrese proves  part of the conjecture for Calabi-Yau threefold stacks satisfying the HL conditions, using similar method of Hall algebra identities in \cite{Calabrese}.  
Note that Bryan and Steinberg \cite{BD} also prove partial result of the crepant resolution conjecture for the DT-invariants. 
We hope that our study of orbifold flop may shed more light on the crepant resolution conjecture. 

An orbifold flop of Calabi-Yau threefold DM stacks is given by the diagram:
\begin{equation}\label{threefold:stack:flop}
\xymatrix{
&\sZ\ar[dl]_{f}\ar[dr]^{f^\prime}&\\
\sX\ar@{-->}[rr]^{\phi}\ar[dr]_{\psi}&& \sX^\prime\ar[dl]^{\psi^\prime}\\
&Y&
}
\end{equation}
where 
\begin{enumerate}
\item $\sX$ and $\sX^\prime$ are smooth Calabi-Yau threefold DM  stacks;
\item $Y$ is a singular variety with only zero-dimensional singularities;
\item Both $\psi$ and $\psi^\prime$ contract cyclic quotients of weighted projective lines 
$\pp(a_1,a_2)$, $\pp(b_1,b_2)$ respectively;
\item $\sZ$ is the common weighted blow-up along the exceptional locus.
\end{enumerate}

\begin{rmk}
Actually the contraction map
$\psi: \sX\to Y$ can be made to have one dimensional singularities as in \cite{Calabrese}.  One still can get some formulas on the DT type invariants, see \cite{Calabrese3}.  Here we only fix to the case of $Y$ with isolated singularities. 
\end{rmk}
Similar to Abramovich-Chen in \cite{AC}, \cite{Chen}, we  prove that  the derived categories of $\sX$ and $\sX^\prime$ are equivalent for such orbifold flops using the idea of  perverse point sheaves of Bridgeland. The equivalence 
\begin{equation}\label{Equivalence:orbifold:flop}
\Phi: D^b(\sX)\to D^b(\sX^\prime)
\end{equation}
is given by the Fourier-Mukai transformation 
$\Phi=\FM$, where 
$$\FM(-)=f^\prime_{\star}(f^\star(-)).$$
Moreover, the equivalence $\Phi$ also sends perverse sheaves to perverse sheaves.
\begin{equation}
\Phi(~^{q}\Per(\sX))=~^{p}\Per(\sX^\prime),
\end{equation}
where $q=-(p+1)$.

Let $\sPA:=\sPA(\sX):=~^{p}\Per(\sX)$.   We work on the Hall algebra $H(\sPA)$ of 
$\sPA$.
Let $K(\sX)$ be the numerical K-group of $\sX$, and 
$$F_0K(\sX)\subset F_1K(\sX)\subset \cdots\subset K(\sX)$$
be the filtration with the support of dimension, see \S \ref{K-theory:basis} for more details. 

Fixing a K-group class $\alpha\in F_1K(\sX)$, let 
$\Hilb^\alpha(\sX)$ be the Hilbert scheme of substacks of $\sX$ with class 
$\alpha$. The DT-invariant is defined by
$$\DT_{\alpha}(\sX)=\chi(\Hilb^\alpha(\sX), \nu_{H}),$$
where $\nu_{H}$ is the Behrend function in \cite{Behrend} of $\Hilb^\alpha(\sX)$. 
Define the $DT$-partition function by 
\begin{equation}\label{DT:partition:function}
\DT(\sX)=\sum_{\alpha\in F_1K(\sX)}\DT_{\alpha}(\sX)q^{\alpha}.
\end{equation}

Similarly the notion of  PT-stable pair for the threefold DM stack $\sX$ is very similar  to Definition \ref{PT-stable:pair}.
A PT-stable pair $[\sO_{\sX}\stackrel{s}{\rightarrow}F]\in D^b(\sX)$ is an object in the derived category such that
$F$ is a pure one-dimensional sheaf supported on curves in $\sX$ with topological data $\beta$, and the cokernel 
$\Coker(s)$ is zero-dimensional. 
Let $\PT^{\beta}(\sX)$ be the PT-moduli space of stable pairs with K-group class $\beta$. 
Then the PT-invariant is defined by: 
$$\PT_{\beta}(\sX)=\chi(\PT^\beta(\sX), \nu_{PT}),$$
where $\nu_{\PT}$ is the Behrend function of $\PT^\beta(\sX)$.
The $\PT$-partition function by
\begin{equation}\label{PT:partition:function}
\PT(\sX)=\sum_{\beta\in F_1K(\sX)}\PT_{\beta}(\sX)q^{\beta}
\end{equation}

Define the following DT-type partition functions:
$$
\DT_{0}(\sX)=\sum_{\alpha\in F_0K(\sX)}\DT_{\alpha}(\sX)q^{\alpha};
$$

$$
\DT_{\exc}(\sX)=\sum_{\substack{\alpha\in F_1K(\sX)/F_0;\\
f_{\star}\alpha=0}}\DT_{\alpha}(\sX)q^{\alpha};
$$

$$
\DT_{\exc}^\vee(\sX)=\sum_{\substack{\alpha\in F_1K(\sX)/F_0;\\
f_{\star}\alpha=0}}\DT_{-\alpha}(\sX)q^{-\alpha};
$$
The main result in the paper is:
\begin{thm}\label{main:flop:introduction}
Let $\phi: \sX\dasharrow \sX^\prime$ be an orbifold  flop of Calabi-Yau threefold DM stacks. Then
$$\Phi_{\star}\left(\DT(\sX)\cdot \frac{\DT^\vee_{\exc}(\sX)}{\DT_0(\sX)}\right)=\DT(\sX^\prime)\cdot \frac{\DT^\vee_{\exc}(\sX^\prime)}{\DT_0(\sX^\prime)},$$
where $\Phi_{\star}$ is understood as sending the data $\alpha\in K(\sX)$ to $\varphi(\alpha)\in K(\sX^\prime)$.
\end{thm}

We prove Theorem \ref{main:flop:introduction} along the method of Bridgeland \cite{Bridgeland11} and Calabrese \cite{Calabrese} by working on the Hall algebra identities in $H(\sPA)$. 
One can define the perverse Hilbert scheme 
$$^{p}\Hilb^\alpha(\sX/Y)$$
which parametrizes the quotients $\sO_{\sX}\to F$ in the category $\sPA$ with fixing class $[F]=\alpha$, since the structure sheaf $\sO_{\sX}\in \sPA$. 
Then we define 
$$\pDT_{\alpha}=\chi(^{p}\Hilb^\alpha(\sX/Y), \nu_{pH})$$
where $\nu_{pH}$ is the Behrend function for $^{p}\Hilb^\alpha(\sX/Y)$.  The partition function is defined by:
$$
\pDT(\sX/Y)=\sum_{\alpha\in F_1K(\sX)}\pDT_{\alpha}(\sX/Y)q^{\alpha}.
$$
We prove that 
$$\pDT(\sX/Y)=\DT(\sX)\cdot \PT^\vee(\sX),$$
where 
$$\PT^\vee(\sX)=\sum_{\beta\in F_1K(\sX)}\PT_{-\beta}(\sX)q^{-\beta}.$$
By Bayer's DT/PT-correspondence for Calabi-Yau threefold  stacks in \cite{Ba}, 
$$\PT^\vee(\sX)=\frac{\DT^\vee_{\exc}(\sX)}{\DT_0(\sX)}.$$
Since for the orbifold flop $\phi: \sX\dasharrow \sX^\prime$, the derived equivalence 
$\Phi$ sends $~^{q}\sA(\sX)$ to $\sPA(\sX^\prime)$, where $q=-(p+1)$. The Theorem follows.  

Our orbifold flops need not to satisfy the HL condition as required by J. Bryan etc in \cite{BCY}.  There exist orbifold flop $\sX\dasharrow\sX^\prime$ of threefold Calabi-Yau stacks such that $\sX$ satisfies the HL condition, while $\sX^\prime$ does not.  The main result in Theorem \ref{main:flop:introduction} implies some information on the Donaldson-Thomas invariants for $\sX^\prime$ form the ones for $\sX$, see \S \ref{section:HLcondition}.

\subsection{Motivic DT-invariants and flops}

The motivic DT-invariants and their wall crossing formula are defined and  studied by Kontsevich and Soibelman in  \cite{KS}, Behrend, Bryan and   Szendroi in \cite{BBS}.  Joyce etc. have a working group on the motivic DT-theory, see \cite{BBJ}, \cite{BJM}.  It is interesting to see how motivic DT-invariants change under orbifold flops.  This can be taken as the motivic analogue of the Crepant Transformation Conjecture in DT-theory. 

Following the method in this paper, we need an integration map from the motivic Hall algebra to the motivic quantum torus, defined in \cite{KS}.  This is related to  Conjecture 4.2 of Kontsevich and Soibelman in \cite{KS}, which can be taken as the motivic version of the Behrend function identities.  We follow the proposal of Joyce in \cite{JS}, and Bridgeland in \cite{Bridgeland11} for the motivic version of the DT invariants.  In \cite{Jiang4}, we will prove  the motivic version of the Joyce-Song formula for the Behrend function identities, and we will also prove that there is a Poisson algebra homomorphism from  the motivic Hall algebra to the motivic quantum torus. Thus the  expected formula for motivic DT-invariants under flops should be true.  We hope to address these conjectures in the future.

\subsection{Comparation to the Gromov-Witten invariants under flops}

DT-invariants have deep connections to Gromov-Witten (GW) invariants via  the GW/DT-correspondence  in \cite{MNOP1}, \cite{MNOP2}. This conjecture has been proved in many cases, including toric threefolds in \cite{MOOP}, and quintic threefolds in \cite{PP}. 

For two birational Calabi-Yau stacks, the crepant transformation conjecture (CTC) says that the partition functions of their GW invariants are related by the analytic continuation.  Let $\sX\dasharrow \sX^\prime$ be a toric crepant birational transformation given by a toric wall crossing. 
In \cite{CIJ}, the authors prove the genus zero CTC.  Using Givental's quantization, in \cite{CI}, Coates and Iritani solved the higher genus CTC.  For such a toric flop, their derived categories are equivalent, and the kernel for the Fourier-Mukai transform is given by the common blow-up.  In \cite{CIJ}, the authors prove that the Fourier-Mukai transform matches the analytic continuation of the quantum connections for $\sX$ and $\sX^\prime$.  Since applying twice the Fourier-Mukai transform, one gets the monodromy for the K-theory, and hence the monodromy of  the derived category, the CTC implies that the monodromy given by the Fourier-Mukai transform  is the same as the monodromy given by the quantum connections.   More general orbifold flops are studied in \cite{CJL}.

Recall for the  orbifold flop of Calabi-Yau threefold DM stacks,  the Fourier-Mukai transform preserves the perverse sheaves for $\sX$ and $\sX^\prime$.  Applying twice the Fourier-Mukai transform gives the Seidel-Thomas twist \cite{ST} for the derived category.  It is interesting to study how the Fourier-Mukai transform can relate DT-invariants and GW-invariants together using the method in this paper and the calculation in \cite{CIJ}. 

\subsection{Outline}
The brief outline of the paper is as follows.  We introduce the orbifold flops for Calabi-Yau threefold  DM stacks in \S \ref{orbifold-flop}.   In \S \ref{section:perverse} we talk about the perverse sheaves on the Calabi-Yau threefold DM stacks and prove the derived equivalence for the orbifold flops.  This generalizes the results as in \cite{AC} and \cite{Bridgeland}.  We also define the counting invariants in the derived category and form the partition functions of the invariants.  In \S \ref{section:motivic:Hall:algebra} we review the motivic Hall algebra of Joyce \cite{Joyce07} and Bridgeland \cite{Bridgeland10}, and define the integration map.  We prove Theorem \ref{main:flop:introduction} in \S \ref{section:Hall:algebra:identity} using the method of Bridgeland and Calabrese on Hall algebra identities.  Finally in \S \ref{section:HLcondition} we talk about the HL condition for the orbifold flop.

\subsection{Acknowledgement}
We would like to thank Tom Coates, Alessio Corti and Richard Thomas for the encouragements and support when the author was staying at Imperial College London where this project was started.   We thank Tom Bridgeland and John Calabrese for the valuable discussions on perverse sheaves and motivic Hall algebras. 
This work is partially supported by  Simons Foundation Collaboration Grant 311837.


\section{Orbifold flop of three dimensional Calabi-Yau DM stacks.}\Label{orbifold-flop}

We define the orbifold flop for three dimensional Calabi-Yau stacks. 

\subsection{The local construction.}

In this section we give the local construction of orbifold flop in three dimensional Calabi-Yau  orbifolds or Deligne-Mumford (DM) stacks. 

Fix $\mathbf{a}=(a_0, a_1)$ and $\mathbf{b}=(b_0, b_1)$ as positive integers.
Let $\bP(a_0,a_1)$, $\bP(b_0,b_1)$ be the corresponding  weighted projective lines.
To avoid gerbe structure we require $\gcd(a_0, a_1,b_0, b_1)=1$.  To preserve the Calabi-Yau property we require $a_0+a_1=b_0+b_1$.
We will call such condition  the \emph{Calabi--Yau condition}.

Recall that in   \cite{Kaw} Kawamata defines the construction of so called "toric flops".  We briefly explain the construction here.   Let $\mathbb{C}^{*}$ acts on the affine variety $\mathbb{A}=\mathbb{A}^{4}$ 
by:
\begin{equation}\label{action}
\lambda(x_0,x_1,y_0,y_1)=(\lambda^{a_0}x_0,\lambda^{a_1}x_1,\lambda^{-b_0}y_0,\lambda^{-b_1}y_1).
\end{equation}
Consider the following stack quotients:
$$
\begin{array}{c}
\widetilde{\sX}=[(\mathbb{A}\setminus \{x_0 =x_1=0\})/\mathbb{C}^{*}]; \\ \\
\widetilde{\sX}^\prime=[(\mathbb{A}\setminus \{y_0=y_1=0\})/\mathbb{C}^{*}]; \\ \\
\widetilde{\sY}=[\mathbb{A}/\mathbb{C}^{*}]=\mbox{spec} R^{\mathbb{C}^{*}},
\end{array}
$$
where $R=\mathbb{C}[x_0,x_1,y_0,y_1]$.
Let $\widetilde{Y}$ be the coarse moduli space of $\widetilde{\sY}$. 
There is a diagram of threefold flops with quotient singularities:
\begin{equation}\label{diagram1}
\xymatrix{
~&~\widetilde{\sZ}\ar[dl]_{f}\ar[dr]^{f^\prime}\\
~ \widetilde{\sX}\;\ar[rd]_{\psi}&&
~\widetilde{\sX}^{\prime}\;\ar[ld]^{\psi'}\\
&~\widetilde{Y},
}
\end{equation} \normalsize
where $\widetilde{\sZ}$ is the fibre product.  The morphism $\psi, \psi'$ are projective and birational whose exceptional loci
$\widetilde{Z}, \widetilde{Z}^\prime$ are isomorphic to the weighted projective lines $\mathbb{P}(a_0,a_1)$ and $\mathbb{P}(b_0,b_1)$ respectively.  The DM stack  $\widetilde{\sZ}$ is  the common blow-up of $\widetilde{\sX}, \widetilde{\sX}^\prime$ along $\widetilde{Z}\subset \sX, \widetilde{Z}^\prime\subset \sX^\prime$ respectively.

If all the $a_i$ and $b_i$ are one, this is the local model of the famous Atiyah flop or the conifold flop. 
We are interested in three dimensional orbifolds, which are $\qq$-Gorenstein algebraic varieties with quotient singularities. 
If $\sX$  is a Calabi-Yau threefold with terminal singularities,  by Kollar 
\cite{Kollar}, the flop $\sX^\prime$ of $\sX$ and the contraction $Y$ all have terminal singularities. The flopping curves are always $\pp^1/\mu_{n}$, where $\mu_n$ is a cyclic group of order $n$ acting on $\pp^1$ by rotation. This is due to the fact that a terminal singularity inside $Y$ is isolated, which is  a hypersurface singularity, and is deformation equivalent to the quotient $\cc^3/\mu_n$ with action by $(1,-1,r)$, where $(r,n)=1$.
We put this construction into the toric picture of Kawamata.

Let $\mu_n$ act on $\sY$ by $\zeta(x_0,x_1,y_0,y_1)=(\zeta x_0,\zeta^{-1}x_1,\zeta^{r}y_0,y_1)$, where 
$(n,r)=1$. 

\begin{rmk}
Note that in the list of singularities as in \cite{Kollar}, $\widetilde{Y}$ has the singularity of type 
$$Y=\spec\cc[x_0,x_1,y_0,y_1]/(x_0y_0-x_1y_1),$$
which is called the conifold singularity in physics.   The singularity can also be understood as the hypersurface singularity:
$$x_0y_0-x_1^2+y_1^2=0.$$ 
Let $Y_n$ be the hypersurface in $\cc^4$ defined by 
$$x_0y_0-x_1^{2n}+y_1^2=0.$$
When $n=1$, $Y_1$ is the above conifold singularity.   The above construction (\ref{diagram1}) also works for this type of singularities when $n\geq 2$, where after we take the resolutions $\widetilde{\sX}$ and $\widetilde{\sX}^\prime$, the exceptional locus are 
$\pp^1$ with normal bundles $\sO_{\pp^1}\oplus \sO_{\pp^1}(-2)$. The singularity in \cite{Kollar} is 
$$Y_1/\mu_n,$$
and the action is given by $\zeta(x_0,x_1,y_0,y_1)=(\zeta x_0,\zeta^{-1}x_1,\zeta^{r}y_0,y_1)$ with $(r,n)=1$. 
\end{rmk}

\begin{defn} \label{local}
A local orbifold flop is given by the following diagram of  stack quotients:
\begin{equation}\label{diagram2}
\xymatrix{
&\sZ=[\widetilde{\sZ}/\mu_n]\ar[dl]_{f}\ar[dr]^{f^\prime}&\\
 \mathcal{X}=[\widetilde{\sX}/\mu_n]\;\ar[rd]_{\psi}&&
\mathcal{X}^{\prime}=[\widetilde{\sX}^\prime/\mu_n]\;\ar[ld]^{\psi'}\\
&Y=\overline{[\widetilde{Y}/\mu_n]},}
\end{equation} \normalsize
where $\overline{[\widetilde{Y}/\mu_n]}$ is the coarse moduli space of $[\widetilde{Y}/\mu_n]$. 
The morphism $\psi, \psi'$ are projective and birational whose exceptional loci
$Z, Z'$ are isomorphic to the weighted projective lines $\mathbb{P}(a_0,a_1)/\mu_n$ and $\mathbb{P}(b_0,b_1)/\mu_n$ respectively.  
\end{defn}

If $a_0,a_1$ are coprime, then the weighted projective line $\pp(a_0,a_1)$ has only two singular points 
$[1,0]$ and $[0,1]$ and the quotient $\pp(a_0,a_1)/\mu_n$ is a toric orbifold in the sense of \cite{BCS}, \cite{Jiang}.
The two singular points 
$[1,0]$ and $[0,1]$ will have local orbifold groups $\mu_{a_0 n}$ and $\mu_{a_1 n}$.

If $a_0,a_1$ are not coprime, then the weighted projective line $\pp(a_0,a_1)$ is a $\mu_d$-gerbe over 
$\pp(\frac{a_0}{d},\frac{a_1}{d})$, where $d=\gcd(a_0,a_1)$. The weighted projective line $\pp(\frac{a_0}{d},\frac{a_1}{d})$ has two singular points 
$[1,0]$ and $[0,1]$ and the quotient $\pp(a_0,a_1)/\mu_n$ is a toric Deligne-Mumford stack in the sense of \cite{BCS},  \cite{Jiang}.
The two singular points 
$[1,0]$ and $[0,1]$ will also have local orbifold groups $\mu_{a_0 n}$ and $\mu_{a_1 n}$, but the local action on it and the normal bundle are quite different comparing to the previous case.

\subsection{Orbifold flop for threefold stacks.}

In this section we establish the general definition of flops of Calabi-Yau threefold  stacks. 
Let $X$ be a quasi-projective $\qq$-Gorenstein Calabi-Yau variety.  We denote by 
$\sX$ the covering Deligne-Mumford stack of $X$. As in \cite{AC}, the stack $\sX$ is a quotient stack 
$$\sX=[P_{X}/\cc^*],$$
where $P_{X}=\spec(\oplus_{i\in\zz}K_{X}^{i})$.

\begin{defn} \label{orbifold:flop}
We say that two smooth Calabi-Yau threefold DM stacks is an
\emph{orbifold flop} $\phi: \mathcal{X} \dashrightarrow \mathcal{X}'$ 
 if they fit into the following 
commutative diagram
$$
\xymatrix{
&E\subset\sZ\ar[dl]_{f}\ar[dr]^{f^\prime}&\\
Z \subset \mathcal{X}\;\ar[rd]_{\psi}& &Z'\subset\mathcal{X}^{'}\;\ar[ld]^{\psi'}\\
& {p} \in Y,}
$$
such that 
\begin{itemize}
  \item $Z \cong \bP(\mathbf{a})/\mu_n$ and $Z' \cong \bP(\mathbf{b})/\mu_n$;
  \item the normal bundle $N_Z$ is isomorphic to 
    $(\oplus_i \sO_{\bP(\mathbf{a})}(- b_i))/\mu_n$ and
    $N_{Z'}$ is isomorphic to $(\oplus_i \sO_{\bP(\mathbf{b})}(- a_i))/\mu_n$;
  \item $\psi$ and $\psi'$ are birational (small) contractions such that
    the exceptional loci $Z$ and $Z'$ map to the point $p$;
\end{itemize}
\end{defn}

\begin{defn} \label{orbifold:flop:ab}
We say that two smooth Calabi-Yau threefold DM stacks is an
\emph{orbifold flop} $\phi: \mathcal{X} \dashrightarrow \mathcal{X}'$ 
of type $(\mathbf{a}, \mathbf{b})$ if they fit into the following 
commutative diagram
$$
\xymatrix{
&E\subset\sZ\ar[dl]_{f}\ar[dr]^{f^\prime}&\\
Z \subset \mathcal{X}\;\ar[rd]_{\psi}& &Z'\subset\mathcal{X}^{'}\;\ar[ld]^{\psi'}\\
& {p} \in Y,}
$$
such that 
\begin{itemize}
  \item $Z \cong \bP(\mathbf{a})$ and $Z' \cong \bP(\mathbf{b})$;
  \item the normal bundle $N_Z$ is isomorphic to 
    $(\oplus_i \sO_{\bP(\mathbf{a})}(- b_i))$ and
    $N_{Z'}$ is isomorphic to $(\oplus_i \sO_{\bP(\mathbf{b})}(- a_i))$;
  \item $\psi$ and $\psi'$ are birational (small) contractions such that
    the exceptional loci $Z$ and $Z'$ map to the point $p$;
\end{itemize}
\end{defn}

\begin{rmk}
The orbifold flop in Definition \ref{orbifold:flop:ab} is defined and studied in \cite{CJL}, where the authors consider the general 
$\pp(a_1,\cdots,a_r)$-flop for higher dimensional DM stacks. 
\end{rmk}

\subsection{Hard Lefschetz condition}

Recall that a DM stack $\sX$ satisfies Hard Lefschetz (HL) condition, if  the  age for a group element is equal to the age of  its inverse.

\begin{prop}\label{HL:condition}
Let $\phi: \sX\dasharrow \sX^\prime$ be an orbifold flop of type $(\mathbf{a}, \mathbf{b})$. 
In order for both $\sX$ and $\sX^\prime$ to satisfy hard Lefschetz (along the zero section), it is necessary and sufficient that $a_i=b_i$ for all $i=0,1$ after reordering. We call this type of orbifold flop the HL orbifold flop.
\end{prop}
\smallskip
\begin{pf}
This is a three dimensional case of the more general quasi-simple orbifold flop defined in \cite{CJL}.  The result is a special case of a more general result there.  On the other hand, one can directly check that for any 
element $v\in \mu_{a_i}, \mu_{b)i}$ for $i=1,2$, the age $\age(v)$ is the same as the age $\age(v^{-1})$, which is the requirement of the HL condition. 
\end{pf}

\section{Perverse coherent sheaves and the derived equivalence.}\label{section:perverse}
\subsection{Perverse coherent sheaves.}
Fix a smooth Calabi-Yau threefold stack $\sX$, denote by $\sA:=\Coh(\sX)$ the abelian category of coherent sheaves over $\sX$. Let 
$D^b(\sX):=D(\sA)=D^{b}(\Coh(\sX))$ be  the bounded derived category of coherent sheaves over $\sX$. The abelian category $\Coh(\sX)$ is the heart of  the standard $t$-structure of $D^{b}(\Coh(\sX))$.

Let  $\phi: \mathcal{X} \dashrightarrow \mathcal{X}'$ be an orbifold flop of Calabi-Yau  threefold stacks,  i.e. there exists  a commutative diagram
$$
\xymatrix{
Z \subset \mathcal{X}\;\ar[rd]_{\psi}& &Z'\subset\mathcal{X}^{'}\;\ar[ld]^{\psi'}\\
& {p} \in Y.}
$$

This orbifold flop satisfies the following properties:
\begin{enumerate}
\item $\psi$ and $\psi^\prime$ are proper, birational and an isomorphism in codimension one;
\item $Y$ is  projective and only has zero dimensional singular locus;
\item the dualising sheaf of $Y$ is trivial, i.e. $\omega_{Y} =\sO_{Y}$;
\item $R\psi_{*} \sO_{\sX}=\sO_{Y}$; $R\psi^\prime_{*} \sO_{\sX^\prime}=\sO_{Y}$;
\item  $\dim_{\qq}N^{1}(\sX/Y)_{\qq}=1$, so is  $\dim_{\qq}N^{1}(\sX^\prime/Y)_{\qq}$,
\end{enumerate}
where $N^{1}(\sX/Y)_{\qq}=N^{1}(\sX/Y)_{\zz}\otimes\qq$ and 
$N^1(\sX/Y)$ is the group of divisors on $\sX$ modulo numerical equivalence over $Y$.  Similar results hold for 
$N^1(\sX^\prime/Y)$.

\textbf{Perverse $t$-structure on $\sX$:}
Let 
$$\pi: \sX\to X$$
be the map to its coarse moduli space, so that we have the following diagram:
\[
\xymatrix{
\sX\ar[d]_{\pi}\ar[dr]^{\psi}& \\
X\ar[r]^{\overline{\psi}}& Y.
}
\]

As in \cite{AC} there are two sub-categories of $D^b(\sX)$:
$$
\begin{cases}
B=\{L\pi^\star C\in D^b(\sX)| C\in D^b(X)\};\\
C_2=\{C\in D^b(\sX)| R\pi_{\star}C=0\}.
\end{cases}
$$
The pair $(B,C_2)$ gives a semiorthogonal decomposition on $D^b(\sX)$. 
On the category $C_2$, there is a standard $t$-structure which is induced from the standard $t$-structure
on $D^b(\sX)$.

Recall from \cite{Bridgeland}, for the map 
$\overline{\psi}: X\to Y$, there is a perverse $t$-structure $t(-1)$ and the heart of this $t$-structure is denoted by 
$\Per^{-1}(X/Y)$. 

\begin{defn}
The derived functor $R\pi_{\star}$ has right adjoint $\pi^{!}$ and the left adjoint $L\pi^\star$. 
Denote by $t(p,0)$ the $t$-structure obtained by gluing: the perverse $t$-structure $t(p)$ on $D^b(X)$, and the standard $t$-structure on $C_2$.  We denote by the heart of this $t$-structure by $\Per^{p}(\sX/Y):=\Per^{p,0}(\sX/Y)$. 
Usually we take $p=-1, 0$ and we always denote by 
$\Per(\sX/Y):=\Per^{-1}(\sX/Y)$. 
\end{defn}

Recall that in \cite{AC},  the perverse sheaf is classified as follows:
An object $E$ in $D^b(\sX)$ is a ``perverse sheaf" i.e. $E\in \Per(\sX/Y)$ if:
\begin{enumerate}
\item $R\pi_{\star}E$ is a perverse sheaf for $\overline{\psi}: X\to Y$ and $\pi: \sX\to X$ is the map to its coarse moduli space;
\item $\Hom(E,C)=0$ for all $C$ in $C_2^{>0}$ and $\Hom(D, E)=0$ for all $D$ in $C_2^{<0}$.
\end{enumerate}
Then Lemma 3.3.1 of \cite{AC} classifies all perverse coherent sheaves:
\begin{lem}
An object $E\in D^b(\sX)$ is a perverse sheaf if and only if  the following conditions are satisfied:
\begin{enumerate}
\item $H_i(E)=0$ unless $i=0$ or $1$;
\item $R^1\psi_{\star}H_0(E)=0$ and $R^0\psi_{\star}H_1(E)=0$;
\item $\Hom(\pi_\star H_0(E),C)=0$ for any sheaf 
$C$ on $\sX$ satisfying  $\overline{\psi}_{\star}C=R^1\psi_{\star}C=0$;
\item $\Hom(D, H_1(E))=0$ for any sheaf  $D$ in $C_2$.
\end{enumerate}
\end{lem}

Recall that in \cite{Bridgeland}, \cite{AC}, the perverse sheaves can be obtained by tilting a torsion pair.  
We say that an object $E\in D(\sA)$ connects to $C_2$, denoted by  $E|C_2$ if $E$ satisfies the conditions: 
$\Hom(E,C)=0$ for all $C$ in $C_2^{>0}$ and $\Hom(D, E)=0$ for all $D$ in $C_2^{<0}$.
Let 
$$\sC=\{E\in \Coh(X)| R\overline{\psi}_{\star}E=0\}$$ and 
let
$$
\begin{array}{l}
\sTT=\{T\in \sA| R^1\overline{\psi}_{\star}(R\pi_{\star}T)=0; T|C_2 \};\\
\ssT=\{T\in \sA| R^1\overline{\psi}_{\star}(R\pi_{\star}T)=0, \Hom(T,\sC)=0,  T|C_2\};\\
\sFF=\{F\in \sA| R^0\overline{\psi}_{\star}(R\pi_{\star}T)=0; \Hom(\sC, F)=0,  F|C_2 \};\\
\ssF=\{F\in \sA| R^0\overline{\psi}_{\star}(R\pi_{\star}T)=0;  F|C_2\}
\end{array}
$$
Then $(^{p}\sT, ^{p}\sF)$ is a torsion pair on $\sA$ for $p=-1,0$ and a tilt of $\sA$ with respect to the torsion pair is the category of perverse coherent sheaves $\sPA:=\Per^{p}(\sX/Y)$.   Then every element $E\in \sPA$ fits into the exact sequence:
\begin{equation}\label{exact:sequence:sPA}
F[1]\hookrightarrow E\twoheadrightarrow T
\end{equation}
with $F\in \sPF$ and $T\in \sPT$.

From Bridgeland \cite{Bridgeland} and Abramovich-Chen \cite{AC}, the category of perverse sheaves forms a heart of 
$t$-structure on $D^b(\sX)$. 
Usually there are actually two perversities $p=-1, 0$.

\subsection{Derived equivalence}
Let $\phi: \sX\dasharrow \sX^\prime$ be an orbifold flop,  in this section we prove, following the method  of \cite{Bridgeland}, \cite{AC},  that there is an equivalence between derived categories:
\begin{equation}
\Phi: D^b(\sX)\to D^b(\sX^\prime)
\end{equation}
by the Fourier-Mukai transformation and 
$$\Phi(\Per^{-1}(\sX/Y))=\Per^0(\sX^\prime/Y).$$

\subsubsection{Perverse point ideal sheaves}

\begin{defn}\label{perverse:ideal:sheaf}
A perverse ideal sheaf $F\in\sPA$ is a sheaf such that it fits into the exact sequence
$$0\rightarrow F\longrightarrow \sO_{\sX}\longrightarrow E\rightarrow 0$$
in $\sPA$.  The object $E$ is called the ``perverse structure sheaf".  A perverse point sheaf
is a perverse structure sheaf such that it is numerically equivalent to the structure sheaf of a point.  
\end{defn}

We have a similar proposition as in \cite[Lemma 3.3.3]{AC}.
\begin{prop}\label{prop:perverse:ideal:sheaf}
A perverse ideal sheaf is a sheaf. A sheaf $F\in\Coh(\sX)$ is a perverse ideal sheaf if and only if it satisfies the following conditions:
\begin{enumerate}
\item $R\pi_{\star}F$ is a perverse ideal sheaf of $\overline{f}: X\to Y$;
\item $\Hom(D, F)=0$ for any sheaf $D\in C_2$. 
\end{enumerate}
\end{prop}

Perverse point sheaves and perverse point-ideal sheaves are simple objects, which satisfy the following properties.
Let $E_1, E_2$ be two perverse point sheaves. Then 
\[
\Hom(E_1,E_2)=
\begin{cases}
0,& E_1\ncong E_2;\\
\cc, & E_1\cong E_2.
\end{cases}
\]
Similarly let $F_1, F_2$ be two perverse point-ideal sheaves. Then 
\begin{equation}\label{point-ideal:sheaves:simple}
\Hom(F_1,F_2)=
\begin{cases}
0,& F_1\ncong F_2;\\
\cc, & F_1\cong F_2.
\end{cases}
\end{equation}

\subsubsection{Moduli of perverse point sheaves}

Let 
$$\sM(\sX/Y): \mbox{Sch}\to \mbox{Sets}$$
be the functor that sends a scheme $S$ to the set of equivalence classes of families of perverse point sheaves parametrized by 
$S$. The functor $\sM(\sX/Y)$ can be taken as the moduli functor of equivalence classes of perverse point-ideal sheaves. 
From (\ref{point-ideal:sheaves:simple}), the automorphism groups of perverse point-ideal sheaves are $\cc^\star$. 
Then the moduli functor  $\sM(\sX/Y)$ is represented by a fine moduli space  $M(\sX/Y)$.  As in \cite[Lemma 4.1.1]{AC}, the moduli space $M(\sX/Y)$ is separated. 

Let $W\subset M(\sX/Y)$ be the distinguished component, which is birational to $Y$.  We want to prove that $W$ is isomorphic to the smooth DM stack $\sX^\prime$ in the orbifold flop diagram:
\[
\xymatrix{
&\sZ\ar[dl]_{f}\ar[dr]^{f^\prime}&\\
\sX\ar@{-->}[rr]&&\sX^\prime.
}
\]

\begin{prop}\label{morphismXprimeW}
There exists a birational morphsim 
$\sX^\prime\to W$.
\end{prop}
\begin{proof}
We construct a family of perverse point sheaves over $\sX^\prime$. 
The candidate for such a family is $\sZ$.  But $\sZ$ in this case contains an extra embedded component and we take the reduction of 
$\sZ$ by removing this component. 

It is sufficient to work on the local model in Diagram (\ref{diagram2}) of Definition \ref{local}.  In this case 
$\sZ=\widetilde{\sZ}/\mu_n$, where $\widetilde{\sZ}=\sO_{\pp(\mathbf{a})\times \pp(\mathbf{b})}(-1,-1)$.
The  reduction  $\sZ_{\red}$ is $\sZ$ modulo the exceptional locus $\pp(\mathbf{a})$. 
We show that the structure sheaf 
$\sO_{\sZ_{\red}}$ is a family of perverse point sheaves over $\sX^\prime$. 
Let 
$$\id\times\pi: \sX^\prime\times_{Y}\sX\to \sX^\prime\times_{Y}X$$
be the natural morphism.  We check that 
$(\id\times \pi)_{\star}I_{\sZ_{\red}}$ is a perverse ideal sheaf. This is the Condition (1) in Proposition \ref{prop:perverse:ideal:sheaf}. 

To check Condition (2) in Proposition \ref{prop:perverse:ideal:sheaf}, we need to prove that 
$\Hom(D, I_{\sZ_{\red}})=0$ for any $D\in C_2$. 
We use the method in \cite{AC}. Let 
$$p: \widetilde{\sX}\to \sX$$
be the finite morphism as in Definition \ref{local}, which taken as a base change.  We argue that 
$\Hom(D, I_{\sZ})=0$. 
Let 
$$p: \widetilde{\sX}\times_{\widetilde{Y}}\widetilde{\sX}^\prime
\to \widetilde{\sX}\times_{Y}\sX^\prime\hookrightarrow \widetilde{\sX}\times\sX^\prime$$
be the corresponding morphisms, where the first is finite, and the second is an embedding. 
Let the image be $T$.  Then we have
$$0\to I_{T}\rightarrow \sO_{\widetilde{\sX}\times\sX^\prime}\rightarrow \sO_{T}\to 0.$$
Let $i: p\hookrightarrow \sX^\prime$ be a point.  We prove that 
$i^\star I_{T}$ has torsion with support in pure dimension one and it can not have sections at the 
preimages of the stacky points of $\sX$ under $p$.  So 
$\Hom(D, I_{T})=0$ for any $D\in C_2$. 
\end{proof}

\subsubsection{The derived equivalence}

Since $W$ is the distinguished component of $M(\sX/Y)$, the universal perverse point sheaf $\mathcal{E}$ gives a diagram:
$$
\xymatrix{
D^b(W)\ar[rr]^{\Phi}\ar[dr]&& D^b(\sX)\ar[dl]\\
&D^b(Y)&
}
$$
To prove that 
$W\cong \sX^\prime$ and 
 $\Phi$ is an equivalence,  we already know from Proposition \ref{morphismXprimeW} there is a birational morphism 
 $\sX^\prime\to W$, we follow the lines in 
\cite[\S 6-7]{BKR} to prove that $W$ is smooth, $W\cong \sX^\prime$ and $\Phi$ is an equivalence sending perverse sheaves to perverse sheaves. We omit the details. 

\begin{example}
In stead of proving  the tedious construction as in \S 3 of \cite{AC}, and \S6, \S 7 of \cite{BKR}, we give an example of orbifold flop. 
Let 
$$\phi: \sX=\pp_{\pp(2,2)}(\sO(-1)\oplus \sX(-3))\dasharrow \sX^\prime=\pp_{\pp(1,3)}(\sO(-2)\oplus \sX(-2))$$
be an example of the local model.  For notational reason let $\sC=\pp(2,2)$ and $\sC^\prime=\pp(1,3)$ be the exceptional locus of 
$\sX$ and $\sX^\prime$ respectively, which are contracted to the singular point $P\in Y$.  Geometrically we can construct the flop 
$\sX^\prime$ as follows.  We can do weighted blow-up of $\sX$ along the exceptional locus $\sC$ and then blowing-down another exceptional curve to get $\sX^\prime$. 

From \cite{Bridgeland}, \cite{AC}, let $y\in \sC$ be a point such that 
$\overline{y}\in C:=\pp^1$ is its image in the maps between coarse moduli spaces:
$$
\xymatrix{
\sC\ar[r]^{\pi}\ar@{^{(}->}[d]& C\ar@{^{(}->}[d]\\
\sX\ar[r]^{\pi}& X.
}
$$

Then we have an exact sequence on $X$:
\begin{equation}\label{flop:exact:sequence1}
0\rightarrow \sO_{C}(-1)\longrightarrow \sO_{C}\longrightarrow \sO_{\overline{y}}\rightarrow 0,
\end{equation}
pulling back to $\sX$ we have the following exact sequence:
\begin{equation}\label{flop:exact:sequence2}
0\rightarrow \sO_{\sC}(-2)\longrightarrow \sO_{\sC}\longrightarrow \sO_{y}\rightarrow 0.
\end{equation}
As in \cite{AC} and \cite{Bridgeland}, the coherent sheaves $\sO_{\sC}(-1)$,  $\sO_{\sC}(-2)$ are not perverse, hence the exact sequences (\ref{flop:exact:sequence1}) and (\ref{flop:exact:sequence2}) do not define exact sequences in $\Per(\sX/Y)$. 
But the shifted ones $\sO_{\sC}(-1)[1]$,  $\sO_{\sC}(-2)[1]$ are  perverse sheaves,  and  we have:
\begin{equation}\label{flop:exact:sequence3}
0\rightarrow \sO_{\sC}\longrightarrow  \sO_{y}\longrightarrow \sO_{\sC}(-2)[1]\rightarrow 0.
\end{equation}
This makes $\sO_{y}$ is not stable in $\Per(\sX/Y)$. 
So the  flopping $\sX^\prime\to Y$ means that we can replace the exceptional curve 
$\sC$ by $\sC^\prime$ so that it parameterizes the extension 
\begin{equation}\label{flop:exact:sequence4}
0\rightarrow \sO_{\sC}(-2)[1]\longrightarrow  \mathcal{E}\longrightarrow \sO_{\sC}\rightarrow 0,
\end{equation}
which is stable in $\Per(\sX/Y)$.  
The moduli stack of perverse point sheaves $W=M(\sX/Y)$ parameterizes perverse point sheaves $\mathcal{E}$ on 
$\sX$.  Geometrically $\sX^\prime$ is obtained by replacing $\sC$ parameterizing the exact sequence (\ref{flop:exact:sequence3}) by 
$\sC^\prime$ parameterizing the exact sequence (\ref{flop:exact:sequence4}). 
\end{example}

\subsection{Moduli of perverse ideal sheaves.}

\subsubsection{$K$-theory class}\label{K-theory:basis}
Let $\sX$ be the smooth Calabi-Yau DM stack and $K_0(\sX)$ the Grothendieck group of $K$-theory with compactly support. 
Recall that in \cite{BCY}, two $F_1, F_2\in K_0(\sX)$ are numerically equivalent, i.e.
$$F_1\sim_{\tiny\mbox{num}}F_2$$
if 
$$\chi(E\otimes F_1)=\chi(E\otimes F_2)$$
for all locally free sheaves $E$ on $\sX$.  Recall that there is a Chern character map
$$\widetilde{\Ch}: K_0(\sX)\to H_{\tiny\mbox{CR}}^*(\sX)$$
from the K-group of $\sX$ to the Chen-Ruan cohomology of $\sX$, 
such that
$$\chi(F)=\int_{I\sX}\widetilde{\Ch}(F)\cdot \widetilde{Td}(\sX).$$
So ``numerical equivalence" means that their associated Chow group classes are the same. 
Let 
$$K(\sX):=K_0(\sX)/\sim_{\tiny\mbox{num}}.$$
There is a natural filtration 
$$F_0(K(\sX))\subset F_1(K(\sX))\subset \cdots \subset K(\sX)$$
which is given by the dimension of the support of coherent sheaves.

\subsubsection{Hilbert scheme of sub-stacks.}
Let $\alpha\in K(\sX)$. We define $\Hilb^{\alpha}(\sX)$ to be the category of families of sub-stacks 
$\sZ\subset \sX$ having $[\sO_{\sZ}]=\alpha$.  From \cite{BCY}, \cite{OS},  $\Hilb^{\alpha}(\sX)$ is represented by a scheme which we still denote it by $\Hilb^{\alpha}(\sX)$. 
Let $\mathcal{I}_{\sZ}$ be the ideal sheaf of $\sZ$ in $\sO_{\sX}$, then we can take $\Hilb^{\alpha}(\sX)$ to be the moduli space of ideal sheaves $\mathcal{I}_{\sZ}$ with $[\sO_{\sZ}]=\alpha$.  In the case that $\sX$ is a smooth scheme, this is the original $\DT$-moduli space, see \cite{Thomas}, \cite{MNOP1}.

\subsubsection{Stable pairs}\label{stable:pair}
For the Calabi-Yau threefold stack $\sX$, generalizing the definition of Pandharipande-Thomas \cite{PT}, a stable pair 
$[\sO_{\sX}\stackrel{s}{\longrightarrow}F]$ is an object in $D^b(\sX)$, such that 
\begin{enumerate}
\item $\dim\Supp(F)\leq 1$ and $F$ is pure;
\item $\Coker(s)$ is zero dimensional.
\end{enumerate}

The stable pairs lies in the heart of a $t$-structure constructed in \cite{Bridgeland11}. 
As in \cite{Bridgeland11}, let 
$$\sP:=\Coh_{0}(\sX)\subset \sA:=\Coh(\sX)$$
be the sub-category consisting of sheaves supported on dimension zero.  Let 
$$\sQ=\{E\in\sA| \Hom(P,E)=0 ~\text{for}~ P\in\sP\}.$$
Then $(\sP, \sQ)$ is a torsion pair:
\begin{enumerate}
\item  if $P\in \sP$ and $Q\in \sQ$, then $\Hom_{\sA}(P,Q)=0$;
\item Every $E\in \sA$ fits into a short exact sequence
$$0\rightarrow P\longrightarrow E\longrightarrow Q\rightarrow 0$$
with $P\in\sP$ and $Q\in \sQ$. 
\end{enumerate}
A new $t$-structure on $D^b(\sX)=D(\sA)$ is defined by tilting the standard $t$-structure, see \S 2.2 of \cite{Bridgeland11}, or \cite{HRS}. The heart $\sA^{\#}$ of this new $t$-structure is given by:
$$\sA^{\#}=\{E\in D(\sA)| H_0(E)\in \sQ, H_1(E)\in \sP, H_i(E)=0 \text{`for~} i\notin\{0,1\}\}.$$
We have $\sQ=\sA\cap\sA^{\#}$ and $\sO_{\sX}\in \sA^{\#}$. Bridgeland \cite{Bridgeland11} proves the following result:
\begin{prop}
A stable pair $[\sO_{\sX}\stackrel{s}{\longrightarrow}F]$ is a epimorphism $\sO_{\sX}\twoheadrightarrow F$ in $\sA^{\#}$ with 
$\dim\Supp(F)\leq 1$ and $F\in \sQ$. 
\end{prop}

Fixing $[\sO_{F}]=\beta\in K(\sX)$, let $\PT^\beta(\sX)$ be the moduli stack of stable pairs, parameterizing the objects 
$[\sO_{\sX}\stackrel{s}{\longrightarrow}F]$ satisfying the conditions in the definition.  From \cite{Ba}, it is represented by a scheme $\PT^\beta(\sX)$.

\subsubsection{DT-type invariants}
\begin{defn}
The DT-invariant of $\sX$ in the class $\alpha\in K(\sX)$ is defined by the weighted Euler characteristic 
$$\DT_{\alpha}(\sX)=\chi(\Hilb^{\alpha}(\sX),\nu_{H}),$$
where 
$$\nu_{H}: \Hilb^{\alpha}(\sX)\to \zz$$
is the Behrend function in \cite{Behrend}.  Similarly, 
the PT-invariant of $\sX$ in the class $\beta\in K(\sX)$ is defined by the weighted Euler characteristic 
$$\PT_{\beta}(\sX)=\chi(\Hilb^{\beta}(\sX),\nu_{\PT}),$$
where 
$$\nu_{\PT}: \PT^\beta(\sX)\to \zz$$
is the Behrend function of $\PT^\beta(\sX)$.
\end{defn}

\begin{rmk}
Both $\Hilb^\alpha(\sX)$ and $\PT^\beta(\sX)$ have symmetric obstruction theories in the sense of Behrend \cite{Behrend}.  If $\sX$ is compact, then the invariants defined by virtual fundamental class are the same as weighted Euler characteristic of Behrend, see Theorem 4.18 of \cite{Behrend}. 
\end{rmk}

\subsubsection{Partition function}

Define the $DT$-partition function by 
\begin{equation}\label{DT:partition:function}
\DT(\sX)=\sum_{\alpha\in F_1K(\sX)}\DT_{\alpha}(\sX)q^{\alpha}
\end{equation}
and the $PT$-partition function by
\begin{equation}\label{PT:partition:function}
\PT(\sX)=\sum_{\beta\in F_1K(\sX)}\PT_{\beta}(\sX)q^{\beta}
\end{equation}

The degree zero $DT$-partition function is defined by 
\begin{equation}\label{DT:partition:function:zero}
\DT_0(\sX)=\sum_{\alpha\in F_0K(\sX)}\DT_{\alpha}(\sX)q^{\alpha}
\end{equation}
and the reduced $DT$-partition function by
\begin{equation}\label{reduced:DT:partition:function}
\DT^\prime(\sX)=\frac{\DT(\sX)}{\DT_0(\sX)}.
\end{equation}

\section{The motivic Hall algebra.}\Label{section:motivic:Hall:algebra}
In this section we review the definition and construction of the motivic Hall algebra of Joyce and Bridgeland in \cite{Joyce07}, \cite{Bridgeland10}. Then we review the integration map from the motivic Hall algebra to the ring of functions of the quantum torus. 

\subsection{Motivic Hall algebra}
We briefly review the notion of motivic Hall algebra in \cite{Bridgeland10}, more details can be found in \cite{Bridgeland10}, \cite{Joyce07}. 

\begin{defn}
The Grothendieck ring of stacks  $K(\St/\cc)$ is defined to be the $\cc$-vector space spanned by isomorphism classes of Artin stacks of finite type over $\cc$ with affine stabilizers, modulo the relations:
\begin{enumerate}
\item for every pair of stacks $\sX_1$ and $\sX_2$ a relation:
$$[\sX_1\sqcup\sX_2]=[\sX_1]+[\sX_2];$$
\item for any geometric bijection $f: \sX_1\to \sX_2$, $[\sX_1]=[\sX_2]$;
\item for any Zariski fibrations $p_i: \sX_i\to \sY$ with the same fibers, $[\sX_1]=[\sX_2]$.
\end{enumerate}
\end{defn}
Let $[\aaa^1]=\ll$, the Lefschetz motive.  If $S$ is a stack of finite type over $\cc$, we define the relative Grothendieck ring of stacks $K(\St/S)$ as follows:

\begin{defn}\label{relative:Grothendieck:group}
The relative Grothendieck ring of stacks  $K(\St/\cc)$ is defined to be the $\cc$-vector space spanned by isomorphism classes of morphisms
$$[\sX\stackrel{f}{\rightarrow}S],$$
with $\sX$ an Artin stack over $S$ of finite type with affine stabilizers, modulo the following relations:
\begin{enumerate}
\item for every pair of stacks $\sX_1$ and $\sX_2$ a relation:
$$[\sX_1\sqcup\sX_2\stackrel{f_1\sqcup f_2}{\longrightarrow}S]=[\sX_1\stackrel{f_1}{\rightarrow}S]+[\sX_2\stackrel{f_2}{\rightarrow}S];$$
\item for any diagram:
$$
\xymatrix{
\sX_1\ar[rr]^{g}\ar[dr]_{f_1}&&\sX_2\ar[dl]^{f_2}\\
&S&,
}
$$
where $g$ is a 
geometric bijection, then $[\sX_1\stackrel{f_1}{\rightarrow}S]=[\sX_2\stackrel{f_2}{\rightarrow}S]$;
\item for any pair of Zariski fibrations 
$$\sX_1\stackrel{h_1}{\rightarrow} \sY;\quad \sX_2\stackrel{h_2}{\rightarrow} \sY$$
with the same fibers, and $g: \sY\to S$, a relation 
$$[\sX_1\stackrel{g\circ h_1}{\longrightarrow} S]=[\sX_2\stackrel{g\circ h_2}{\longrightarrow} S].$$
\end{enumerate}
\end{defn}
The motivic Hall algebra  in \cite{Joyce07} and \cite{Bridgeland10} is defined as follows.
Let $\sM$ be the moduli stack of coherent sheaves on $\sX$. It is an algebraic stack, locally of finite type over $\cc$. The motivic Hall algebra is the vector space 
$$H(\sA)=K(\St/\sM)$$
equipped with a non-commutative product given by the role:
$$[\sX_1\stackrel{f_1}{\longrightarrow} \sM]\star[\sX_2\stackrel{f_2}{\longrightarrow} \sM]=[\sZ\stackrel{b\circ h}{\longrightarrow} \sM],$$
where $h$ is defined by the following Cartesian square:
\[
\xymatrix{
\sZ\ar[r]^{h}\ar[d]&\sM^{(2)}\ar[r]^{b}\ar[d]^{(a_1,a_2)}&\sM \\
\sX_1\times\sX_2\ar[r]^{f_1\times f_2}& \sM\times \sM,&
}
\]
with $\sM^{(2)}$ the stack of short exact sequences in $\sA$, and 
the maps $a_1, a_2, b$ send a short exact sequence
$$0\rightarrow A_1\longrightarrow B\longrightarrow A_2\rightarrow 0$$
to sheaves $A_1$, $A_2$, and $B$ respectively. Then $H(\sA)$ is an algebra over 
$K(\St/\cc)$.

\subsection{The integration map}

Recall that in \S 3 of  \cite{Bridgeland10}, there exists maps of commutative rings:
$$K(\Sch/\cc)\to K(\Sch/\cc)[\ll^{-1}]\to K(\St/\cc),$$
where $K(\Sch/\cc)$ is the Grothendieck  ring of schemes of finite type over $\cc$. 
Since $H(\sA)$ is an algebra over $K(\St/\cc)$, define a $K(\Sch/\cc)[\ll^{-1}]$-module
$$H_{\reg}(\sA)\subset H(\sA)$$
to be the span of classes of maps $[X\stackrel{f}{\rightarrow}\sM]$ with $X$ a scheme.  An element of $H(\sA)$ is regular if it lies in $H_{\reg}(\sA)$. The following is Theorem 5.1 of \cite{Bridgeland10}.

\begin{thm}
The sub-module of regular elements of $H(\sA)$ is closed under the convolution product:
$$H_{\reg}(\sA)\star H_{\reg}(\sA)\subset H_{\reg}(\sA)$$
and is a $K(\Sch/\cc)[\ll^{-1}]$-algebra. Moreover, the quotient 
$$H_{\ssc}(\sA)=H_{\reg}(\sA)/(\ll-1)H_{\reg}(\sA)$$
is a commutative $K(\Sch/\cc)$-algebra.
\end{thm}

The algebra $H_{\ssc}(\sA)$ is called the semi-classical Hall algebra. In \cite{Bridgeland10}, Bridgeland also defines a Poisson bracket on $H(\sA)$ by:
$$\{f, g\}=\frac{f\star g-g\star f}{\ll-1}.$$
This bracket preserves the subalgebra $H_{\reg}(\sA)$. 

Let $\Delta\subset F_1K(\sX)$ be the effective cone of $F_1K(\sX)$, that is, the collection of elements of the forms
$[E]$, where $E$ is a one-dimensional sheaf.  Define
$$\cc[\Delta]=\bigoplus_{\alpha\in \Delta}\cc\cdot x^{\alpha}$$
to be the ring generated by symbols $x^\alpha$ for $\alpha\in \Delta$, with product defined by:
$$x^\alpha\star x^{\beta}=(-1)^{\chi(\alpha,\beta)}\cdot x^{\alpha+\beta}.$$  
The ring is commutative since Euler form is skew-symmetric. The Poisson bracket is given by:
$$\{x^{\alpha}, x^{\beta}\}=(-1)^{\chi(\alpha,\beta)}\cdot \chi(\alpha,\beta)\cdot x^{\alpha+\beta}.$$
The following theorem is proved in \S 5.2 of \cite{Bridgeland10}.

\begin{thm}\label{thm:5.2}(Theorem 5.2, \cite{Bridgeland10})
Let $\nu: \sM\to\zz$ be the locally constructible Behrend function. Then there is a Poisson algebra homomorphism:
\begin{equation}\label{integration:map}
I: H_{\ssc}(\sA)\to \cc[\Delta]
\end{equation}
such that 
$$I([\sZ\stackrel{f}{\rightarrow}\sM_{\alpha}])=\chi(\sZ, f^\star\nu)\cdot x^{\alpha}.$$
\end{thm}

\begin{rmk}
The proof of Theorem \ref{thm:5.2} relies on the Behrend function identities in \S 10 of \cite{JS}, which was originally proved for coherent sheaves by Joyce-Song \cite{JS}.  These identities was recently proved by V. Bussi \cite{Bussi} using algebraic method and also works in characteristic $p$, see \cite{Jiang3} for another method using Berkovich spaces. 
In \cite{Jiang4} we will generalize the integration map to the motivic level of the Behrend functions. 
\end{rmk}

\textbf{Integration map for $H(\sPA)$}:

For the abelian category of perverse coherent sheaves $\sPA$, we have a similar definition $H(\sPA)$, the motivic Hall algebra of $\sPA$. 
The semi-classical Hall algebra $H_{\ssc}(\sPA)$ can be similarly defined. 
Let $\sPM$ be the moduli stack of objects in the category $\sPA$.  There is an integration map
\begin{equation}\label{integration:map:sMP}
I: H_{\ssc}(\sPA)\to \cc[\Delta]
\end{equation}
such that 
$$I([\sZ\stackrel{f}{\rightarrow}\sPM_{\alpha}])=\chi(\sZ, f^\star\nu)\cdot x^{\alpha}.$$
Here $\nu: \sPM\to \zz$
is the Behrend function of $\sPM$. 
The proof of this morphism requires the Behrend function identities similar to \S 10 of \cite{JS}, \cite{Bussi}. 
Since the elements in  $\sPA$ are semi-Schur, i.e. for any $E\in\sPA$, 
$\Ext^i(E,E)=0$ for $i<0$,   in \cite{Jiang3} the author proves the Joyce-Song formula for the Behrend function identities  using Berkovich spaces.
Thus the integration map really exists in this case.

\section{DT-invariants identities under flops}\label{section:Hall:algebra:identity}

In this section we study the Hall algebra identities, following \cite{Bridgeland11} and \cite{Calabrese}, and prove the main result.

\subsection{Infinite-type Hall algebras} 
In this section we enlarge the definition of the Hall algebra, as in \S 4.2 of \cite{Bridgeland11} and \cite{Calabrese}. 
For the stack $\sM$, define infinite-type Grothendieck group 
$L(\St_{\infty}/S)$ by the symbols $[\sX\to S]$, but with $\sX$ only assumed to be locally of finite type over 
$S$.  Then we need to drop the relation $(1)$ in Definition \ref{relative:Grothendieck:group}.  The infinite-type Hall algebra is then 
$$H_\infty(\sA)=L(\St_{\infty}/\sM)$$
$$H_\infty(\sPA)=L(\St_{\infty}/\sPM).$$
\begin{rmk}
By working on infinite-type Hall algebra, we may not have integration map $I$ in (\ref{integration:map}) and 
(\ref{integration:map:sMP}).  We will have such an integration map $I$ in the Laurent Hall algebra 
$H_{\Lambda}\subset H_\infty$,  and $H(\sA)\subset H_{\Lambda}$. 
\end{rmk}

\subsection{Perverse Hilbert scheme}

Let $\sA_{\leq 1}\subset \sA$ be the full sub-category consisting of sheaves with support of $\dim\leq 1$.  Similarly, 
$\sPA_{\leq 1}\subset \sPA$ is the full sub-category consisting of perverse sheaves with support of $\dim\leq 1$. Let 
$H_\infty(\sA_{\leq 1})$ ($H_\infty(\sPA_{\leq 1})$) be the corresponding sub-Hall algebra. 

The first element in our formula is 
$$\srH_{\leq 1}\in H_\infty(\sA_{\leq 1}),$$
the Hilbert scheme of $\sX$, which parameterizes quotients 
$$\sO_{\sX}\twoheadrightarrow F$$
in $\sA_{\leq 1}$. 
Let $\sM_{\leq 1}\subset \sM$ be the moduli stack of coherent sheaves with support $\dim\leq 1$. Then 
$\srH_{\leq 1}$ is given by the morphism $\Hilb_{\leq 1}(\sX)\to \sM_{\leq 1}$. 

\begin{rmk}
If $\sO_{\sX}\twoheadrightarrow E$ is a quotient in $\sA_{\leq 1}$, then $E\in \sPT$.  This is because 
$E\in \sPT$, and the quotient of torsion is torsion.  So the morphism 
$$\Hilb_{\leq 1}(\sX)\to \sM_{\leq 1}$$
factors through the element $\srT_{\leq 1}$, which is represented by
$[\sPT\to \sM_{\leq 1}]$.  Hence $\srH_{\leq 1}\in H_{\infty}(\sPA_{\leq 1})$, since 
$\srT_{\leq 1}\in \sPM_{\leq 1}$. 
\end{rmk}

\subsection{Framed coherent sheaves}
Let $\sB\subset \sA$ be a sub-category.  We denote by $\mathds{1}_{\sB}$ the element of $H_\infty(\sA)$ represented by
the inclusion of stacks $\sB\subset \sM$, which is an open immersion. (Similar for $\sA_{\leq 1}$ and $\sPA_{\leq 1}$.)

Following \S 2.3 of \cite{Bridgeland11}, we define  $\sM_{\leq 1}^{\sO}$, the stack of framed coherent sheaves, which parametrizes coherent sheaves with a fixed section $\sO_{\sX}\to E$.  Then $\Hilb_{\leq 1}(\sX)$ is an open subscheme of 
$\sM_{\leq 1}^{\sO}$ by considering a surjective section
$[\sO_{\sX}\twoheadrightarrow E]\in \Hilb_{\leq 1}(\sX)$. We also have a forgetful morphism:
$$\sM_{\leq 1}^{\sO}\to \sM_{\leq 1}$$
by taking 
$[\sO_{\sX}\to E]$ to $E\in \sM_{\leq 1}$.  Given any open substack 
$\sB\subset \sM_{\leq 1}$, we have a Cartesian diagram:
\begin{equation}\label{diagram:framed}
\xymatrix{
\sB^{\sO}\ar[r]\ar[d]&\sM_{\leq 1}^{\sO}\ar[d]\\
\sB\ar[r]&\sM_{\leq 1},
}
\end{equation}
and $\mathds{1}_{\sB}^{\sO}\in H_\infty(\sA_{\leq 1})$. 

Similarly if $\sPB\subset \sPM_{\leq 1}$ is an open stack, then we have similar diagram as in (\ref{diagram:framed})
and an element $\mathds{1}_{\sPB}^{\sO}\in H_\infty(\sPA_{\leq 1})$.

Finally let $\PHilb_{\leq 1}(\sX/Y)$ be the ``perverse Hilbert scheme" parametrizing quotients of $\sO_{\sX}$ in 
$\sPA_{\leq 1}$. Then we have an element  $\srPH_{\leq 1}\in H_{\infty}(\sPA_{\leq 1})$. 

\subsection{Hall algebra identities.}

We prove several Hall algebra identities in this section, following the method in \cite{Calabrese}, \cite{Bridgeland11}. 

\begin{thm}\label{main:Hall:algebra:identity}
We have:
$$\srPH_{\leq 1}\star \mathds{1}_{\sPF[1]}=\mathds{1}_{\sPF[1]}^{\sO}\star \srH_{\leq 1}.$$
\end{thm}
\begin{proof}
First let us analyze both sizes of the equality. The left hand side  (\mbox{LHS}) is represented by a stack 
$\mathfrak{M}_{L}$, parameterizing diagrams:
\[
\xymatrix{
\sO_{\sX}\ar@{->>}[d]&&\\
P_1\ar@{^{(}->}[r]& E\ar@{->>}[r]& P_2
}
\]
where all objects are in $\sPA_{\leq 1}$, the bottom sequence is exact in $\sPA_{\leq 1}$, 
$\sO_{\sX}\twoheadrightarrow P_1$ is surjective in $\sPA_{\leq 1}$, and 
$P_2\in \sPF[1]$. 

The right hand side (\mbox{RHS}) is represented by a stack $\mathfrak{M}_{R}$, parameterizing diagrams:
\[
\xymatrix{
\sO_{\sX}\ar[d]&&\sO_{\sX}\ar[d]^{\sur}\\
F[1]\ar@{^{(}->}[r]& E\ar@{->>}[r]& T
}
\]
where the horizontal sequence 
$$F[1]\hookrightarrow E\twoheadrightarrow T$$
is an exact sequence in $\sPA_{\leq 1}$, and $F\in \sPF$, $T\in \sPT_{\leq 1}$. 
Moreover 
$\sO_{\sX}\to T$ is surjective in $\sA_{\leq 1}$, and has perverse cokernel lying in $\sPF[1]$. 
Actually given a perverse coherent sheaf $E\in \sPA_{\leq 1}$, there exists a unique exact sequence above.

As in \S 3.3 of \cite{Calabrese}, we construct the following diagram:
\begin{equation}\label{main:diagram}
\xymatrix{
\mathfrak{M}_{L}\ar[dr]^{f_{L}}&& \mathfrak{M}\ar[dl]^{f^\prime}\ar[dr]^{g}& &\mathfrak{M}_{R}\ar[dl]^{g_{R}}\\
&\mathfrak{M}^\prime&& \mathfrak{N}
}
\end{equation}
such that the maps are either geometric bijections or Zaraski fibrations. 

We first define the stack $\mathfrak{M}^\prime$, which parametrizes the diagrams of the form:
\[
\xymatrix{
\sO_{\sX}\ar[d]_{\varphi}\\
E
}
\]
such that $\pCoker(\varphi)\in \sPF[1]$.  By Lemma 3.2 of \cite{Calabrese}, this is equivalent to 
$\Cone(\varphi)\in D^{\leq 1}(\sX)$, which is open. So 
$\mathfrak{M}^\prime$ is an open substack of the stack of framed perverse sheaves $\sPM_{\leq 1}^{\sO}$.

The first lemma is:
\begin{lem}\label{map:f}
There is a map 
$f_L: \mathfrak{M}_{L}\to \mathfrak{M}^\prime$ induced by the composition 
$$\sO_{\sX}\twoheadrightarrow P_1\hookrightarrow E,$$ 
which is a geometric bijection. 
\end{lem}
\begin{proof}
The map $f_L:  \mathfrak{M}_{L}\to \mathfrak{M}^\prime$ is an equivalence on $\cc$-points. As we see later, 
$\srPH_{\leq 1}$, $\mathds{1}_{\sPF}$ are all Laurent elements in the Hall algebra $H_\infty(\sPA_{\leq 1})$. So for any 
$\alpha\in K(\sX)$, 
$\mathfrak{M}_{L,\alpha}\to \mathfrak{M}_{\alpha}^\prime$ is of finite type.
\end{proof}

Secondly, we define the stack $\mathfrak{M}$, which parametrizes the diagrams of the form:
\[
\xymatrix{
&\sO_{\sX}\ar[d]^{\phi}&\\
F[1]\ar@{^{(}->}[r]& E\ar@{->>}[r]& T,
}
\]
where the horizontal sequence is a short exact sequence of perverse sheaves and 
$F\in\sPF$, $T\in \sPT_{\leq 1}$, and $\pCoker(\phi)\in\sPF[1]$.
The stack $\mathfrak{M}$ can be understood as a fibre product:
\[
\xymatrix{
\mathfrak{M}\ar[r]\ar[d]&\mathfrak{M}^\prime\ar[d]\\
\sZ\ar[r]&\sPM_{\leq 1},
}
\]
where $\sZ$ is the element $\mathds{1}_{\sPF[1]}\star \mathds{1}_{\sPT_{\leq 1}}$. 

\begin{lem}\label{map:fprime}
The morphism $f^\prime: \mathfrak{M}\to \mathfrak{M}^\prime$ defined by forgetting the bottom exact sequence 
is a geometric bijection. 
\end{lem}
\begin{proof}
As in Proposition 3.4 of \cite{Calabrese}, considering the following diagram:
\[
\xymatrix{
\sZ\ar[r]\ar[d]&\sPM_{\leq 1}^{(2)}\ar[r]^{b}\ar[d]&\sPM_{\leq 1}\\
\sPF[1]\times\sPT_{\leq 1}\ar[r]& \sPM_{\leq 1}\times\sPM_{\leq 1}
}
\]
where the bottom is an open immersion and $b$ is of finite type. The morphism $\sZ\to \sPM_{\leq 1}$ induces an equivalence on $\cc$-points since $(\sPF[1], \sPT)$ is a torsion pair in $\sPA$. As 
$\mathfrak{M}\to \mathfrak{M}^\prime$ is a base change, it is a geometric bijection. 
\end{proof}

So to prove the main identity, we need to prove that 
$$[\mathfrak{M}]=[\mathfrak{M}_{R}]\in H_\infty(\sPA).$$
In Diagram (\ref{main:diagram}), we are only left to define the stack $\mathfrak{N}$. 
The stack $\mathfrak{N}$ is defined as the moduli stack of the following diagrams:
\[
\xymatrix{
&&\sO_{\sX}\ar[d]^{\sur}\\
F[1]\ar@{^{(}->}[r]& E\ar@{->>}[r]& T,
}
\]
where the bottom exact sequence lies in $\sPA$ and $F\in \sPF$, $T\in \sPT_{\leq 1}$. Moreover the morphism
$\sO_{\sX}\to T$ is surjective in $\sA$ and has perverse cokernel in $\sPF[1]$. There exist two maps
$$l: \mathfrak{M}\to \mathfrak{N}$$
which is given by
$$
\xymatrix{
&\sO_{\sX}\ar[d]^{\phi}&\\
F[1]\ar@{^{(}->}[r]& E\ar@{->>}[r]& T,
}\longmapsto
\xymatrix{
&&\sO_{\sX}\ar[d]^{\sur}\\
F[1]\ar@{^{(}->}[r]& E\ar@{->>}[r]& T,
}
$$
and 
$$r: \mathfrak{M}_{R}\to \mathfrak{N}$$
which is given by:
$$
\xymatrix{
\sO_{\sX}\ar[d]&&\sO_{\sX}\ar[d]^{\sur}&\\
F[1]\ar@{^{(}->}[r]& E\ar@{->>}[r]& T,
}\longmapsto
\xymatrix{
&&\sO_{\sX}\ar[d]^{\sur}\\
F[1]\ar@{^{(}->}[r]& E\ar@{->>}[r]& T.
}
$$

\begin{prop}\label{l:r}
The maps $l$ and $r$ are two Zaraski fibrations with the same fibers. 
\end{prop}
\begin{proof}
First over a perverse sheaf $E$, we have 
$$F[1]\hookrightarrow E\twoheadrightarrow T$$
in $\sPA_{\leq 1}$, and $F\in \sPF, T\in\sPT_{\leq 1}$. 
So over an element 
$$
\xymatrix{
&&\sO_{\sX}\ar[d]^{\sur}\\
F[1]\ar@{^{(}->}[r]& E\ar@{->>}[r]& T.
}
$$
in $\mathfrak{N}$, the fiber of $r$ is $\Hom_{\sX}(\sO_{\sX}, F[1])$ and the fiber of $l$ is: 
the lifts 
$$\sO_{\sX}\stackrel{}{\longrightarrow} E$$
which has perverse cokernel $\pCoker(\varphi)\in \sPF[1]$. 

From the exact sequence 
$$0\rightarrow \Hom_{\sX}(\sO_{\sX}, F[1])\longrightarrow  \Hom_{\sX}(\sO_{\sX}, E)\longrightarrow  
\Hom_{\sX}(\sO_{\sX}, T)\rightarrow 0,$$
for a map $\varphi: \sO_{\sX}\to T$, all lifts of $\varphi$ by 
$\sO_{\sX}\to E$ are in bijection with $\Hom_{\sX}(\sO_{\sX}, F[1])$. Then to finish the proof, we have to show that 
any lift of $\sO_{\sX}\to T$ is one $\sO_{\sX}\to F[1]$ such that the perverse cokernel is in 
$\sPF[1]$. 

Let $\varphi: \sO_{\sX}\to T$ be a map with $T\in \sPT_{\leq 1}$ and $\pCoker(\varphi)\in \sPF[1]$. 
Let $\delta: \sO_{\sX}\to E$ be a lift such that 
\[
\xymatrix{
0\ar[r]\ar[d]&\sO_{\sX}\ar[r]^{=}\ar[d]_{\delta}&\sO_{\sX}\ar[d]^{\varphi}\\
F[1]\ar@{^{(}->}[r]& E\ar@{->>}[r]& T.
}
\]
is an exact-sequence diagram. Hence we have an exact sequence on cokernels:
$$F[1]\rightarrow \pCoker(\delta)\rightarrow \pCoker(\varphi)\rightarrow 0.$$
So from Lemma 1.5 of \cite{Calabrese}, $\pCoker(\delta)\in\sPF[1]$. 
This construction works for families and we are done. 
\end{proof}
Hence from Lemmas \ref{map:f}, \ref{map:fprime}, Proposition \ref{l:r},
$$[\mathfrak{M}_{L}]=[\mathfrak{M}_{R}],$$
hence the theorem. 
\end{proof}

\subsection{PT-type invariants identities.}

Recall that in \S \ref{stable:pair} we define a torsion pair
$(\sP, \sQ)$ on $\sA$, where 
$$\sP=\{\text{coherent sheaves supported on dimension zero\}}$$
and $\sQ$ is the right orthogonal of $\sP$.  Recall that the tilt of $\sA$ is given by
$\sA^{\#}$.

The scheme $\Hilb_{\leq 1}^{\#}(\sX)$ parameterizes quotients 
$\sO_{\sX}\twoheadrightarrow F$ in $\sA^{\#}$ supported on dimension $\leq 1$. So we have an element 
$\srH_{\leq 1}^{\#}\in H_\infty(\sA_{\leq 1})$ which gives rise to the PT-stable pair invariants. 

Let $\sQ_{\leq 1}$ be the stack parameterizing objects in $\sQ_{\leq 1}\subset \sM_{\leq 1}$. Then there exists an element 
$\mathds{1}_{\sQ_{\leq 1}}\in H_\infty(\sA_{\leq 1})$.  Its framed version is denoted by $\mathds{1}_{\sQ_{\leq 1}}^{\sO}$, 
parameterizing 
$$\{\sO_{\sX}\to F\}$$
for $F\in \sQ_{\leq 1}$.  Similar to \S 4.5 of \cite{Bridgeland11}, we have the following Hall algebra identity:
\begin{equation}\label{Hall:algebra:identity:PT}
\mathds{1}_{\sQ_{\leq 1}}^{\sO}=\srH_{\leq 1}^{\#}\star \mathds{1}_{\sQ_{\leq 1}}.
\end{equation}

\textbf{Restriction to the exceptional locus.}
Following Calabrese \cite{Calabrese}, we define the following:
\[
\begin{cases}
\sQ_{\exc}=\{Q\in \sQ_{\leq 1}| \dim\Supp R\psi_{\star}Q=0\};\\
\sPA_{\exc}=\{E\in\sPA_{\leq 1}| \dim\Supp R\psi_{\star}E=0\};\\
\sPT_{\exc}=\sPT\cap \sPA_{\exc};\\
\sPT_{\bullet}=\sPT_{\exc}\cap \sQ_{\exc},
\end{cases}
\]
where $\psi: \sX\to Y$ is the contraction map. 
Hence inside $\Hilb_{\leq 1}^{\#}(\sX)$, there is an open subscheme $\Hilb^{\#}_{\exc}(\sX)$, parameterizing quotients 
$\sO_{\sX}\to F$ in $\sA^{\#}_{\leq 1}$ such that $F\in \sPT_{\bullet}$.  Its Hall algebra element is denoted by
$\srH_{\exc}^{\#}\in H_\infty(\sA_{\leq 1})$.

\begin{prop}
We have the following identity in $H_\infty(\sA_{\leq 1})$:
$$\mathds{1}_{\sPT_{\bullet}}^{\sO}=\srH_{\exc}^{\#}\star \mathds{1}_{\sPT_{\bullet}}.$$
\end{prop}
\begin{proof}
First if we have a morphism
$\sO_{\sX}\to T$ in $\sA^{\#}$ with $T\in \sPT_{\bullet}$, then we have a sequence
$\sO_{\sX}\to I\to T$ in $\sA^{\#}$, where $I$ is the image in $T$. 
From Lemma 2.3 of \cite{Bridgeland11}, $I$ is a sheaf, so $\sO_\sX\twoheadrightarrow I$ has cokernel
$P\in \sP$. 
Considering 
$$I\to T\to Q,$$
where $Q$ is the quotient.  The short exact sequence 
$I\hookrightarrow T\twoheadrightarrow Q$ lies in $\sA$, $Q\in\sPT$ since it is a quotient of $T$, and 
$Q\in \sQ$ since it is an object in $\sA^{\#}$. Also $R\psi_{\star}Q$ supports on dimension zero since 
$R\psi_{\star}T$ is.  So $Q\in \sPT$. 

Conversely, let $\sO_{\sX}\to I$ be an element in $\Hilb_{\exc}^{\#}$, where it is an epimorphism in $\sA^{\#}$. Let 
$$I\hookrightarrow T\twoheadrightarrow Q$$
be an exact sequence of coherent sheaves, with $I\in \sQ_{\exc}$, $Q\in\sPT_{\bullet}$. So 
$T\in\sPT_{\bullet}$. Also $T\in\sQ_{\exc}$, so we need to prove $I\in \sPT$.

Considering the exact sequence 
$$\sO_{\sX}\to I\twoheadrightarrow P,$$
with $P$ supported in dimension zero. Let 
$$I\twoheadrightarrow F$$
be the projection to the torsion free part of $I$ with respect to $(\sPT, \sPF)$.  Then the morphism
$$\sO_{\sX}\to I\twoheadrightarrow F$$
is zero, since $F\in \sPT$ has no sections.  Thus there exists a morphism 
$$P\to F$$
such that 
$$I\twoheadrightarrow P\to F=I\twoheadrightarrow F.$$
But $P$ is a skyscraper sheaf, which implies that $P\to F=0$. 
So $I\twoheadrightarrow F=0$, which implies that $F=0$ and $I\in\sPT$. 
The RHS and LHS are given by the following correspondence:
$$
\xymatrix{
\sO_{\sX}\ar[d]&&&\\
I\ar@{^{(}->}[r]& T\ar@{->>}[r]& Q,
}\longmapsto
\xymatrix{
\sO_{\sX} \ar[r]& T.
}
$$
which is a bijection on $\cc$-points. 
\end{proof}

\subsection{Duality functor.}

We briefly recall the duality functor 
\begin{equation}\label{duality:functor}
\dd: D^b(\sX)\to D^b(\sX)
\end{equation}
defined by:
$$E\mapsto R\Hom_{\sX}(E, \sO_{\sX})[2].$$
This duality functor satisfies the following property:
\begin{equation}\label{duality:functor:preserve:equality}
\dd(\sQT_{\bullet})=\sPF,
\end{equation}
where $q=-(p+1)$. 
\begin{rmk}
Since $\sX$ is a smooth Calabi-Yau threefold stack, the proof of (\ref{duality:functor:preserve:equality}) is very similar to Lemma 3.7 of \cite{Calabrese}. We omit the details. 
\end{rmk}

Let $\dd^\prime:=\dd[1]$ be the functor of $\dd$ shifted by one. 

\begin{prop}\label{duality:equality}
We have:
$$\dd^\prime(\mathds{1}_{\sQT_{\bullet}})=\mathds{1}_{\sPF[1]};$$
$$\dd^\prime(\mathds{1}^{\sO}_{\sQT_{\bullet}})=\mathds{1}^{\sO}_{\sPF[1]},$$
where $\mathds{1}_{\sQT_{\bullet}}, \mathds{1}_{\sPF[1]}$ are elements in $H_{\infty}(\sA_{\leq 1})$ given by the stacks 
$\smQT_{\bullet}, \smPF\in \sM_{\leq 1}$; and 
$\mathds{1}^{\sO}_{\sQT_{\bullet}}, \mathds{1}^{\sO}_{\sPF[1]}$ are elements in $H_{\infty}(\sPA_{\leq 1})$ given by the stacks 
$\smQT^{\sO}_{\bullet}, \smPF[1]^{\sO}\in \sM^{\sO}_{\leq 1}$.
\end{prop}
\begin{proof}
Proof is very similar to Proposition 3.8 of \cite{Calabrese}. 
\end{proof}

\begin{prop}\label{Hall:algebra:identity:want:cancel}
The formula in Theorem \ref{main:Hall:algebra:identity} is given by:
$$\srPH_{\leq 1}\star \mathds{1}_{\sPF[1]}=\mathds{1}_{\sPF[1]}\star \dd^\prime(\srH_{\exc}^{\#})\star 
\srH_{\leq 1}.$$
\end{prop}
\begin{proof}
The formula in Theorem \ref{main:Hall:algebra:identity} is:
$$\srPH_{\leq 1}\star \mathds{1}_{\sPF[1]}=\mathds{1}^{\sO}_{\sPF[1]}\star  
\srH_{\leq 1}.$$
From Proposition \ref{duality:equality}, 
$$\mathds{1}^{\sO}_{\sPF[1]}=\dd^\prime(\mathds{1}^{\sO}_{\sQT_{\bullet}})
=\dd^\prime(\srH_{\leq 1}^{\#}\star \mathds{1}_{\sQT_{\bullet}})
=\mathds{1}_{\sPF[1]}\star \dd^\prime(\srH_{\exc}^{\#}).$$
\end{proof}

\subsection{Laurent elements and a complete Hall algebra.}

As in \cite{Bridgeland11} and \cite{Calabrese}, we need to introduce Laurent elements in the numerical Grothendieck group $K(\sX)$.  The reason to do this is that the infinite-type Hall algebra $H_\infty(\sA_{\leq 1})$ is too big to support an integration map and we have to work on spaces of locally finite type. 

Recall that for the contraction $\psi: \sX\to Y$, we have 
$$N_1(\sX/Y)\hookrightarrow N_1(\sX)\twoheadrightarrow N_1(Y).$$
So we have 
$$N_1(\sX)=N_1(\sX/Y)\oplus N_1(Y).$$
We can index elements in $N_{\leq 1}(\sX)=N_1(\sX)\oplus N_0(\sX)=N_1(Y)\oplus N_1(\sX/Y)\oplus N_0(\sX)$ by $(\gamma, \delta, n)$. 
Recall that we have a Chern character map:
$$[E]\in F_1K(\sX)\mapsto (\Ch_2(E), \Ch_3(E))\in N_1(\sX)\oplus N_0(\sX).$$
Let $\pDelta\subset F_1K(\sPA)\cong N_1(\sX)\oplus N_0(\sX)$ be the image of the Chern character map  of $\sPA_{\leq 1}$.  Then the Hall algebra 
$H(\sPA_{\leq 1})$ is graded by $\pDelta$.  Let 
$\mathscr{C}\subset N_1(\sX/Y)$ be the effective curve classes in $\sX$ contracted by $\psi$. 

\begin{defn}
Let $L\subset \pDelta$ be a subset. We call $L$ to be {\em Laurent} is the following conditions hold:
\begin{enumerate}
\item  for any $\gamma$, there exists an $n(\gamma, L)$ such that for all 
$\delta, n$, with $(\gamma, \delta, n)\in L$, we have $n\geq n(\gamma, L)$;
\item  for all $\gamma, n$, there exists a $\delta(\gamma, n, L)\in \mathscr{C}$, such that 
for all $\delta$ with $(\gamma, \delta, n)\in L$ one has $\delta\leq \delta(\gamma, n, L)$.
\end{enumerate}
\end{defn}

Let $\Lambda$ be the set of all Laurent subsets of $\pDelta$.  The set $\Lambda$ satisfies the following properties as in Lemma 3.10 of  \cite{Calabrese}:
\begin{enumerate}
\item If $L_1, L_2\in \Lambda$, then $L_1+L_2\in\Lambda$;
\item If $\alpha\in\pDelta$ and  $L_1, L_2\in \Lambda$, then there exist only finitely many decompositions
$\alpha=\alpha_1+\alpha_2$ with $\alpha_i\in L_i$.
\end{enumerate}

\textbf{The $\Lambda$-completion $H(\sPA_{\leq 1})_{\Lambda}$.}

Recall the algebra:  
$$\cc_{\sigma}[\pDelta]=\bigoplus_{\alpha\in\pDelta}x^{\alpha}.$$
The integration map is given by:
$$I: H_{\ssc}(\sPA_{\leq 1})\to \cc_{\sigma}[\pDelta].$$
For any $\pDelta$-graded associative algebra $R$, the $\Lambda$-completion 
$R_{\Lambda}$ is defined to be the vector space of formal series:
$$\sum_{(\gamma,\delta,n)}x_{(\gamma,\delta,n)}$$
with $x_{(\gamma,\delta,n)}\in R_{x_{(\gamma,\delta,n)}}$, and $x_{(\gamma,\delta,n)}=0$ outside a Laurent subset. The product is defined by:
$$x\cdot y=\sum_{\alpha\in\pDelta}\sum_{\alpha_1+\alpha_2=\alpha}x_{\alpha_1}\cdot y_{\alpha_2}.$$
Then the integration map $I: H_{\ssc}(\sPA_{\leq 1})\to \cc_{\sigma}[\pDelta]$ induces a morphism on the completions:
$$I_{\Lambda}: H_{\ssc}(\sPA_{\leq 1})_{\Lambda}\to \cc_{\sigma}[\pDelta]_{\Lambda}.$$

\textbf{Elements in $H(\sPA_{\leq 1})_{\Lambda}$.}

Let $\mathfrak{S}$ be an algebraic stack of locally of finite type over $\cc$, such that 
$[\mathfrak{S}\to \sPM_{\leq 1}]$ is a map to $\sPM_{\leq 1}$.  For $\alpha\in\pDelta$, the preimage of $\sPM_{\alpha}$ is denoted by $\mathfrak{S}_{\alpha}$. The element 
$$[\mathfrak{S}\to \sPM_{\leq 1}]\in H_{\infty}(\sPA_{\leq 1})$$
is \textbf{Laurent} if $\mathfrak{S}_{\alpha}$ is a stack of finite type for all $\alpha\in\pDelta$, and $\mathfrak{S}_{\alpha}$ is empty for $\alpha$ outside a Laurent subset. 

Then following results are due to Calabrese in \cite{Calabrese}.

\begin{prop}
The elements 
$$\mathds{1}_{\sPF[1]},\quad \mathds{1}^{\sO}_{\sPF[1]}, \quad \srPH_{\leq 1}, \quad \srH_{\leq 1}$$
are all Laurent.
\end{prop}
\begin{proof}
The proof of the result is very similar to Proposition 3.13, 3.14, and 3.15 of \cite{Calabrese}.

The Laurentness of $\mathds{1}_{\sPF[1]},\quad \mathds{1}^{\sO}_{\sPF[1]}$ is from the fact that once fixing numerical data $(\gamma,\delta,n)$, Riemann-Roch tells us that the subset $\alpha$ is bounded. 
That the element $\srPH_{\leq 1}$ is Laurent comes from a detail analysis that once we fix $\gamma, n$, varying 
$\delta$ then the corresponding perverse Hilbert scheme is of finite type.  The case of $\srH_{\leq 1}$ is from the Hall algebra identity:
$$\srPH_{\leq 1}\star \mathds{1}_{\sPF[1]}=\mathds{1}^{\sO}_{\sPF[1]}\star  
\srH_{\leq 1}$$ 
in Theorem \ref{main:Hall:algebra:identity}.
\end{proof}

\textbf{Duality functor revisited.}
Recall the duality functor in  (\ref{duality:functor}), and the shifted duality functor 
$\dd^\prime=\dd[1]$.  By Proposition \ref{duality:equality}, 
\begin{equation}\label{Dprime:sending}
T\in \sQT_{\bullet}\mapsto \dd^\prime(T)\in \sPF,
\end{equation}
where $T$ has numerical data $(0,\delta,n)$, while $\dd^\prime(T)$ has numerical data $(0,-\delta,n)$.

\subsection{Proof of the main results.}

First we have the following Hall algebra identity from Proposition \ref{Hall:algebra:identity:want:cancel}:
\begin{equation}\label{Hall:algebra:identity:need}
\srPH_{\leq 1}\star \mathds{1}_{\sPF[1]}=\mathds{1}_{\sPF[1]}\star \dd^\prime(\srH_{\exc}^{\#})\star 
\srH_{\leq 1}.
\end{equation}
We need to cancel $\mathds{1}_{\sPF[1]}$ in (\ref{Hall:algebra:identity:need}).  The elements 
$\srPH_{\leq 1}, \quad \srH_{\exc}^{\#}, \quad \srH_{\leq 1}$ are all regular in $H_{\ssc}(\sPA_{\leq 1})$, but 
$\mathds{1}_{\sPF[1]}$ is not. To overcome this difficulty, we us Joyce's stability result, as done by Bridgeland \cite{Bridgeland11} and \cite{Calabrese}.  Recall that elements in $\sPF[1]$ will have numerical data 
$(0,\delta,n)$, for $n\geq 0$.  We need the fact 
$$(\ll-1)\cdot \log(\mathds{1}_{\sPF[1]})\in H_{\reg}(\sPA_{\leq 1}),$$
which can be done by introducing stability condition on the objects that have numerical data $(0,\delta,n)$. 
This means that we work in the category $\sPA_{\exc}$. 
Define a stability condition $\mu$ by:
\[
(0,\delta,n)\mapsto
\begin{cases}
1, & \delta\geq 0;\\
2, & \delta<0.
\end{cases}
\]
The stability condition $\mu$ is a weak stability condition in sense of Definition 3.5 of \cite{JS}. 

\begin{lem}
The set of $\mu$-semistable objects of slope 
$\mu=2$ is $\sPF[1]$, and the set of $\mu$-semistable objects of slope 
$\mu=1$ is $\sPT_{\exc}$.
\end{lem}
\begin{proof}
An object $P$ is said to be semistable if for all proper subobjects $P^\prime\subset P$ we have
$\mu(P^\prime)\leq \mu(P/P^\prime)$.  If 
$P$ is any semistable object, we have the torsion and torsion-free exact sequence:
$$F[1]\hookrightarrow P\twoheadrightarrow T$$
where $F\in\sPF$, $T\in\sPT_{\leq 1}$. If 
$F[1]\neq 0$ and $T\neq 0$, then 
$2=\mu(F[1])\leq \mu(T)=1$ which is impossible. So it must be torsion or torsion free. 
\end{proof}

As in \cite[Proposition 3.18]{Calabrese}, the stability condition $\mu$ is permissible in sense of \cite[Definition 4.7]{Joyce207}. 
The following result is Theorem 6.3, Corollary 6.4 in \cite{Bridgeland11}, Proposition 3.20 of \cite{Calabrese}:
\begin{prop}\label{automorphism:Hall:algebra}
In the complete Hall algebra $H(\sPA_{\leq 1})_{\Lambda}$, we have:
$$\mathds{1}_{\sPF[1]}=\exp(\epsilon),$$
with $\eta=(\ll-1)\cdot \epsilon\in H_{\reg}(\sPA_{\leq 1})_{\Lambda}$ a regular element.  Here the element 
$\epsilon$ is $\log(\mathds{1}_{\sPF[1]})$. The automorphism:
$$\Ad_{\mathds{1}_{\sPF[1]}}: H(\sPA_{\leq 1})_{\Lambda}\to H(\sPA_{\leq 1})_{\Lambda}$$
preserves regular elements and the induced Poisson automorphism of 
$H_{\ssc}(\sPA_{\leq 1})_{\Lambda}$ is given by:
$$\Ad_{\mathds{1}_{\sPF[1]}}=\exp\{\eta,-\}.$$
\end{prop}

\begin{thm}
We have:
$$\pDT(\sX/Y)=I(\dd^\prime(\srH_{\exc}^{\#}))\cdot \DT(\sX). $$
\end{thm}
\begin{proof}
From the Hall algebra identity 
$$
\srPH_{\leq 1}\star \mathds{1}_{\sPF[1]}=\mathds{1}_{\sPF[1]}\star \dd^\prime(\srH_{\exc}^{\#})\star 
\srH_{\leq 1}.$$ 
in (\ref{Hall:algebra:identity:need}) and Proposition \ref{automorphism:Hall:algebra},  we have the equation:
$$\srPH_{\leq 1}=\dd^\prime(\srH_{\exc}^{\#})\cdot \exp\{\eta,-\}\cdot \srH_{\leq 1}.$$
So when applying the integration map and note that the Poisson bracket is trivial when applying the integration map we have:
$$I_{\Lambda}(\srPH_{\leq 1})=I_{\Lambda}(\dd^\prime(\srH_{\exc}^{\#}))\cdot I_{\Lambda}(\srH_{\leq 1}).$$
Hence the result follows, due to $I_{\Lambda}(\srH_{\leq 1})=\DT(\sX)$. 
\end{proof}

\begin{cor}
We have:
$$\pDT(\sX/Y)=\frac{\DT^{\vee}_{\exc}(\sX)}{\DT_0(\sX)}\cdot \DT(\sX).$$
\end{cor}
\begin{proof}
We need to use A. Bayer's DT/PT-correspondence for Calabi-Yau orbifolds in \cite{Ba}.  In \cite{Ba}, Bayer proves that 
$$\DT^{\prime}(\sX)=\PT(\sX).$$
Hence 
$$\DT_{\exc}^{\prime}(\sX)=\PT_{\exc}(\sX).$$
Since $I_{\Lambda}(\srH_{\exc}^{\#})=\PT_{\exc}(\sX)$, 
$$I_{\Lambda}(\dd^\prime(\srH_{\exc}^{\#}))=\PT_{\exc}^{\vee}(\sX).$$
The result follows since $\PT_{\exc}^{\vee}(\sX)=\DT_{\exc}^{\prime,\vee}(\sX)=\frac{\DT^{\vee}_{\exc}(\sX)}{\DT_0(\sX)}$.
\end{proof}

\subsection*{Proof of Theorem \ref{main:flop:introduction}}
\begin{proof}
For an orbifold flop 
$$
\xymatrix{
\sX\ar@{-->}[rr]^{\phi}\ar[dr]&&\sX^\prime\ar[dl]\\
&Y&
}
$$
we have an equivalence:
$$\Phi: D^b(\sX)\to D^b(\sX^\prime)$$
which is given by the Fourier-Mukai transformation.  Moreover
$$\Phi(\Per(\sX/Y))=\Per(\sX^\prime/Y).$$
Then on the Hall algebra $H_{\ssc}(\sPA_{\leq 1})$, we have 
$$\Phi(\srPH_{\leq 1}(\sX))=\srPH_{\leq 1}(\sX^\prime).$$
So $\Phi_{\star}(\pDT(\sX/Y))=\pDT(\sX^\prime/Y)$.
\end{proof}

\section{Discussion on the Hard Lefschetz condition.}\Label{section:HLcondition}

In this section we give a short discussion on the Hard Lefschetz (HL) condition for orbifold flops. 

Proposition \ref{HL:condition} tells us that an orbifold flop $\phi: \sX\dasharrow \sX^\prime$ of  type 
$(a_1, a_2; b_1,b_2)$ satisfies the HL condition if and only if $a_i=b_i$ for $i=1,2$. Our result in Theorem \ref{main:flop:introduction}
may tells the DT-invariants for $\sX$ that does not satisfy the HL condition. 

\begin{cor}
Let $\phi: \sX\dasharrow \sX^\prime$ be an orbifold flop of type $(\mathbf{a}, \mathbf{b})$ satisfying the HL condition. Then 
$$\Phi(\DT(\sX)\cdot \DT^\vee_{\exc}(\sX))= \DT(\sX^\prime)\cdot \DT^\vee_{\exc}(\sX^\prime).$$
\end{cor}
\begin{proof}
Theorem \ref{main:flop:introduction} gives the formula:
$$\Phi_{\star}\left(\DT(\sX)\cdot \frac{\DT^\vee_{\exc}(\sX)}{\DT_0(\sX)}\right)=\DT(\sX^\prime)\cdot \frac{\DT^\vee_{\exc}(\sX^\prime)}{\DT_0(\sX^\prime)}.$$
If $\phi$ satisfies the HL condition, then $a_1=b_1, a_2=b_2$.  Hence the local orbifold groups are the same for both 
$\sX$ and $\sX^\prime$.
Degree zero Donaldson-Thomas theory for both $\sX$ and $\sX^\prime$ counts irreducible representations of the local group for  $\sX$ and $\sX^\prime$, then 
$\Phi_{\star}(\DT_0(\sX))=\DT_0(\sX^\prime)$. The Corollary follows. 
\end{proof}

We discuss the case of local picture of orbifold flop of type $(a_1, a_2; b_1,b_2)$ with $\sum_ia_i=\sum_ib_i$. 
Consider the diagram (\ref{diagram1}), 
$\phi: \widetilde{\sX}\dasharrow \widetilde{\sX}^\prime$ is an orbifold flop of type $(a_1, a_2; b_1,b_2)$. 
Both $\widetilde{\sX}$ and 
$\widetilde{\sX}^\prime$ are Calabi-Yau threefold stacks with $A_n$-singularities. 
Bryan, Cadman and Young \cite{BCY} studies the DT-invariants of such Calabi-Yau threefold stacks satisfying the HL condition by the method of  orbifold topology vertex. In some cases, they derived nice formula for the DT-partition functions. 

Our main result implies that using DT-invariants of Calabi-Yau threefold stacks with the HL condition, we may get DT-partition function for Calabi-Yau threefold stacks without the HL conditions. 
\begin{example}
Let $(a_1, a_2; b_1,b_2)=(2,2; 1,3)$. 
Consider 
$$
\sX=\xymatrix{
\sO_{\pp(2,2)}(-1)\oplus \sO_{\pp(2,2)}(-3)\ar@{-->}[rr]^{}\ar[dr]&&\sX^\prime=\sO_{\pp(1,3)}(-2)\oplus \sO_{\pp(1,3)}(-2)\ar[dl]\\
&Y&
}
$$
which is an orbifold flop such that $\sX$ satisfies HL condition, but $\sX^\prime$ does not. 
The stack $\sX$ is a local $B\mu_2$-gerbe over $\pp^1$, and the DT-partition function for $\sX$ was calculated in \cite[\S  4.4]{BCY}.  Our main result Theorem \ref{main:flop:introduction} implies the relationship between the DT-partition function for  $\sX$ and 
the DT-partition function for $\sX^\prime$.
\end{example}

\subsection*{}

\end{document}